\documentclass{article}

\usepackage{microtype}
\usepackage{graphicx}
\usepackage{caption}
\usepackage{subcaption}
\usepackage{booktabs} 
\usepackage{multirow}

\usepackage{hyperref}

\usepackage[accepted]{icml2023}

\usepackage{amsmath}
\usepackage{amssymb}
\usepackage{mathtools}
\usepackage{amsthm}

\usepackage[capitalize,noabbrev]{cleveref}

\theoremstyle{plain}
\newtheorem{theorem}{Theorem}[section]
\newtheorem{proposition}[theorem]{Proposition}
\newtheorem{lemma}[theorem]{Lemma}
\newtheorem{corollary}[theorem]{Corollary}
\theoremstyle{definition}

\newtheorem{assumption}[theorem]{Assumption}
\theoremstyle{remark}
\newtheorem{remark}[theorem]{Remark}

\newcommand{\PP}[1]{\mathbb{P}_{#1}}
\newcommand{\EE}[1]{\mathbb{E}_{#1}}

\definecolor{C1}{RGB}{239,133,54}
\definecolor{C0}{RGB}{57,118,175}
\usepackage[textsize=tiny]{todonotes}

\icmltitlerunning{End-to-End Learning for Stochastic Optimization: A Bayesian Perspective}

\begin{document}

\twocolumn[
\icmltitle{End-to-End Learning for Stochastic Optimization: A Bayesian Perspective}



\icmlsetsymbol{equal}{*}

\begin{icmlauthorlist}
\icmlauthor{Yves Rychener}{epfl}
\icmlauthor{Daniel Kuhn}{epfl}
\icmlauthor{Tobias Sutter}{konstanz}
\end{icmlauthorlist}

\icmlaffiliation{epfl}{Risk Analytics and Optimization Chair, \'Ecole Polytechnique F\'ed\'erale de Lausanne, Switzerland}
\icmlaffiliation{konstanz}{Department of Computer and Information Science, University of Konstanz, Germany}

\icmlcorrespondingauthor{Yves Rychener}{yves.rychener@epfl.ch}

\icmlkeywords{End-to-End Learning, Distributionally Robust Optimization, Stochastic Programming, Posterior Bayes Action}

\vskip 0.3in
]



\printAffiliationsAndNotice{}  

\begin{abstract}
We develop a principled approach to end-to-end learning in stochastic optimization. First, we show that the standard end-to-end learning algorithm admits a Bayesian interpretation and trains a posterior Bayes action map. Building on the insights of this analysis, we then propose new end-to-end learning algorithms for training decision maps that output solutions of empirical risk minimization and distributionally robust optimization problems, two dominant modeling paradigms in optimization under uncertainty. Numerical results for a synthetic newsvendor problem illustrate the key differences between alternative training schemes. We also investigate an economic dispatch problem based on real data to showcase the impact of the neural network architecture of the decision maps on their test performance.
\end{abstract}

\section{Introduction}\label{sec:intro}
Most practical decision problems can be framed as stochastic optimization models that minimize the expected value of a loss function impacted by one's decisions and by an exogenous random variable~$Y$. Stochastic optimization techniques are routinely used, for example, in portfolio selection \cite{markowitz2000mean} or economic dispatch and pricing \cite{cournot1897researches,wong2007pricing} among many other areas. However, the probability distribution~$\PP{Y}$ of the uncertain problem parameter~$Y$ is generically unknown and needs to be inferred from finitely many training samples $\widehat{Y}_1,\ldots,\widehat{Y}_N$. In {\em empirical risk minimization} (ERM), $\PP{Y}$ is simply replaced with the empirical (uniform) distribution on the given samples. Unfortunately, ERM is susceptible to overfitting, that is, it leads to decisions that exploit artefacts of the training samples but perform poorly on test data. This effect is also referred to as the ``optimizer's curse" in decision theory \cite{smith2006optimizer}. Various regularization techniques have been proposed to combat this effect. {\em Distributionally robust optimization} (DRO) \cite{delage2010distributionally,wiesemann2014distributionally}, for example, seeks decisions that are worst-case optimal in view of a large family of distributions that could have generated the training samples. Alternatively, any additional information or beliefs about $\mathbb P_{Y}$ may be used as a prior that is updated upon observing training data like in Bayesian estimation. The use of prior information also has a regularizing effect, and the resulting optimal decision is referred to as the {\em posterior Bayes action} \cite{berger2013bayesian_opt}.
ERM and DRO implicitly assume that {\em independent} samples from $\mathbb P_Y$ form the {\em only} source of information available to the decision-maker. However, this assumptions often fails to hold in practice. Financial asset returns are not stationary, and their distribution is correlated with slowly varying macroeconomic factors~\cite{li2002macroeconomic}. Similarly, the distribution of wind energy production levels depends on meteorological conditions and is strongly correlated with wind speeds.  \textit{Contextual stochastic programs} \cite{bertsimas2020predictive} exploit contextual information such as macroeconomic factors or wind speeds to inform decision-making. Existing approaches to contextual stochastic optimization can be grouped into two categories. \textit{Predict-then-optimize} approaches (see, e.g., \cite{mivsic2020data} for a survey) first use some method from statistics or machine learning to learn a parametric or non-parametric model of the distribution~$\mathbb P_Y$. In a second step, the learned distribution model is used in an optimization model to compute a decision. A key drawback of this approach is that the machine learning method used for predicting~$\mathbb P_Y$ (e.g., maximum likelihood estimation) is agnostic of the downstream optimization model. Some ideas to remedy this shortcoming are discussed in \cite{ref:Elmachtoub-22}.
\textit{End-to-end} learning approaches, on the other hand, train a decision map directly from the available data without the detour of first estimating a model for~$\mathbb P_Y$ \cite{donti2017end2end_constr}.



In this work, we investigate the usage of neural networks in end-to-end learning. The relation between a trained neural network and the corresponding Bayes optimal classifier~\cite{baum1987supervised,wan1990neural,papoulis2002probability,kline2005revisiting} as well as the universal approximation capabilities of neural networks~\cite{cybenko1989approximation, NIPS2017_32cbf687} are well understood in traditional regression and classification. In the context of end-to-end learning, however, similar results are lacking. In this work, we aim to close this gap.

\paragraph{Related Work:}
The concept of \textit{end-to-end learning} for stochastic optimization was first introduced by \citet{donti2017end2end_constr}. Instead of minimizing a generic loss function such as the mean-square error, end-to-end learning directly minimizes the task loss over a class of neural networks that embed the underlying stochastic optimization model in their architecture. Despite their unorthodox structure, such neural networks can be differentiated by exploiting the Karush-Kuhn-Tucker optimality conditions of the embedded optimization model \cite{cvxpylayers2019}. Hence, they remain amenable to gradient-based training schemes. More recent approaches to end-to-end learning relax the requirement that the neural network must contain an optimization layer and use simpler architectures to approximate the decision map \cite{ref:Zhang-20,ref:Butler-21,ref:Uysal-21}. This paper complements these efforts. Instead of introducing new algorithms to improve performance, we aim to advance our theoretical understanding of the existing algorithms and extend them to broader problem classes. 

 The term {\em end-to-end} is used slightly differently across different machine learning communities. We adopt the same convention as \citet{donti2017end2end_constr} whereby a decision model is ``end-to-end" if it is trained on the task loss. In this case the feature extractor and prescriptor are trained jointly and directly on the task loss of interest. In traditional deep learning, on the other hand, end-to-end learning usually refers to the training of a deep network without hand-crafted features that processes the data input in its original format, such as audio spectrograms \cite{amodei2016deep}, images~\cite{wang2011end,he2016deep} or text \cite{wang2017tacotron}.


End-to-end learning is also related to \textit{reinforcement learning}, which seeks a policy that maps an observable state to an action in a {\em dynamic} decision-making context. Reinforcement learning agents have been successfully trained to play board games \cite{silver2017mastering, silver2018general} as well as video games \cite{vinyals2019grandmaster} or to steer self-driving cars \cite{kiran2021deep}. However, a key distinguishing feature of reinforcement learning applications is their dynamic nature. The decision-maker interacts with an unknown environment over multiple periods, and the action chosen in a particular period affects the state of the environment in the next period. The decision-maker thus seeks an action that not only incurs a small loss in the current period but also leads to a favorable state in the next period. 
In contrast, end-to-end learning focuses on {\em static} decision problems. We highlight that end-to-end learning is also related to \textit{contextual bandits} \cite{chu2011contextual,ref:book:Csaba}, where an agent sequentially chooses among a finite set of actions, whose expected loss or payoff depends on an unknown distribution conditioned on an observable context. End-to-end learning differs from contextual bandits and reinforcement learning in that training is performed offline. Only the immediate loss after training is relevant, which eliminates the notorious exploration-exploitation trade-off \cite{audibert2009exploration,graves2014towards}.




\paragraph{Contributions:} Our contributions are summarized below. \\
    $\bullet$ We develop a general and versatile modeling framework for end-to-end learning in stochastic optimization.\\
    $\bullet$ We show that the widely used standard algorithm for end-to-end learning outputs a posterior Bayes action map.\\
    $\bullet$ Leveraging the insights of our Bayesian analysis, we propose new end-to-end learning algorithms for training decision maps that output solutions of ERM and DRO problems.\\
    $\bullet$ We show that existing universal approximation results for neural networks extend to decision maps of end-to-end learning models with projection and optimization layers.\\
    $\bullet$ We experimentally compare different approaches to end-to-end learning and different types of contextual information in the framework of a newsvendor problem with synthetic data and an economic dispatch problem with real data.

\paragraph{Notation:}
All random objects are defined on an abstract probability space $(\Omega,\mathcal{F},\PP{})$, and $\EE{}[\cdot]$ denotes the expectation with respect to $\PP{}$. Random objects are denoted by capital letters and their realizations by the corresponding lowercase letters. Given two random vectors~$X$ and~$Y$, we use $\PP{Y}$ and $\PP{Y|X}$ to denote the marginal distribution of~$Y$ and the conditional distribution of~$Y$ given~$X$, respectively.

\section{End-to-End Learning}\label{sec:endtoend}

We consider a decision problem impacted by a random vector $Y\in\mathcal{Y}$, and we assume that the decision-maker has access to an observation $X\in\mathcal{X}$ that provides information about the distribution of~$Y$. In order to express all possible causal relationships between~$Y$ and~$X$, we further assume that there is an unobservable confounder~$Z\in\mathcal{Z}$. For example, $Z$ could represent a parameter that uniquely determines the joint distribution~$\PP{(X,Y)}$ of~$X$ and~$Y$. All Bayesian network structures of interest are shown in Figure~\ref{fig:causal-model}. In Figure~\ref{fig:bayesian:standard}, $X$ and $Y$ have the same parent and no cross influence. This is the case, for instance,  if, conditional on~$Z$, $X=[\widehat Y_1,\ldots,\widehat Y_N]$ consists of multiple independent and identically distributed (i.i.d.) copies of $Y$. In Figure~\ref{fig:bayesian:measurement}, there is an additional direct causal link from~$Y$ to~$X$. This is the case, for instance, if~$X$ represents a noisy measurement of~$Y$. Finally, in Figure~\ref{fig:bayesian:markov}, there is an additional direct causal link from~$X$ to~$Y$. This is the case, for instance, if~$Y$ represents the current state and~$X$ the previous state of a Markov chain or if~$X$ captures contextual information. 

The decision maker aims to solve the stochastic program
\begin{equation}\label{eq:SP}
\min_{a\in \mathcal{A}} \EE{}[\ell(Y, a)|X].
\end{equation}
\begin{figure}
    \centering
    \begin{subfigure}[b]{0.3\linewidth}
    \centering
    \tikzset{every picture/.style={line width=0.75pt}} 

\begin{tikzpicture}[x=0.75pt,y=0.75pt,yscale=-0.6,xscale=0.6]

\draw   (113,72) .. controls (113,63.72) and (126.43,57) .. (143,57) .. controls (159.57,57) and (173,63.72) .. (173,72) .. controls (173,80.28) and (159.57,87) .. (143,87) .. controls (126.43,87) and (113,80.28) .. (113,72) -- cycle ;
\draw   (154,121) .. controls (154,112.72) and (167.43,106) .. (184,106) .. controls (200.57,106) and (214,112.72) .. (214,121) .. controls (214,129.28) and (200.57,136) .. (184,136) .. controls (167.43,136) and (154,129.28) .. (154,121) -- cycle ;
\draw   (56,121) .. controls (56,112.72) and (69.43,106) .. (86,106) .. controls (102.57,106) and (116,112.72) .. (116,121) .. controls (116,129.28) and (102.57,136) .. (86,136) .. controls (69.43,136) and (56,129.28) .. (56,121) -- cycle ;
\draw   (43,173) .. controls (43,164.72) and (63.15,158) .. (88,158) .. controls (112.85,158) and (133,164.72) .. (133,173) .. controls (133,181.28) and (112.85,188) .. (88,188) .. controls (63.15,188) and (43,181.28) .. (43,173) -- cycle ;
\draw   (99,212.5) .. controls (99,202.28) and (119.15,194) .. (144,194) .. controls (168.85,194) and (189,202.28) .. (189,212.5) .. controls (189,222.72) and (168.85,231) .. (144,231) .. controls (119.15,231) and (99,222.72) .. (99,212.5) -- cycle ;
\draw    (117,80) -- (91.47,103.64) ;
\draw [shift={(90,105)}, rotate = 317.2] [color={rgb, 255:red, 0; green, 0; blue, 0 }  ][line width=0.75]    (10.93,-3.29) .. controls (6.95,-1.4) and (3.31,-0.3) .. (0,0) .. controls (3.31,0.3) and (6.95,1.4) .. (10.93,3.29)   ;
\draw    (170,80) -- (188.75,103.44) ;
\draw [shift={(190,105)}, rotate = 231.34] [color={rgb, 255:red, 0; green, 0; blue, 0 }  ][line width=0.75]    (10.93,-3.29) .. controls (6.95,-1.4) and (3.31,-0.3) .. (0,0) .. controls (3.31,0.3) and (6.95,1.4) .. (10.93,3.29)   ;
\draw    (180,135) -- (160.64,192.11) ;
\draw [shift={(160,194)}, rotate = 288.73] [color={rgb, 255:red, 0; green, 0; blue, 0 }  ][line width=0.75]    (10.93,-3.29) .. controls (6.95,-1.4) and (3.31,-0.3) .. (0,0) .. controls (3.31,0.3) and (6.95,1.4) .. (10.93,3.29)   ;
\draw    (86,138) -- (86,156) ;
\draw [shift={(86,158)}, rotate = 270] [color={rgb, 255:red, 0; green, 0; blue, 0 }  ][line width=0.75]    (10.93,-3.29) .. controls (6.95,-1.4) and (3.31,-0.3) .. (0,0) .. controls (3.31,0.3) and (6.95,1.4) .. (10.93,3.29)   ;
\draw    (87,188) -- (103.42,200.77) ;
\draw [shift={(105,202)}, rotate = 217.87] [color={rgb, 255:red, 0; green, 0; blue, 0 }  ][line width=0.75]    (10.93,-3.29) .. controls (6.95,-1.4) and (3.31,-0.3) .. (0,0) .. controls (3.31,0.3) and (6.95,1.4) .. (10.93,3.29)   ;

\draw (143,72) node [anchor=center][inner sep=0.75pt]    {\tiny$Z \in \mathcal{Z }$};
\draw (86,121) node [anchor=center][inner sep=0.75pt]    {\tiny$X\in \mathcal{X}$};
\draw (184, 121) node [anchor=center][inner sep=0.75pt]    {\tiny$Y\in \mathcal{Y}$};
\draw (88, 173) node [anchor=center][inner sep=0.75pt]    {\tiny$A=m( X)$};
\draw (144,212.5) node [anchor=center][inner sep=0.75pt]    {\tiny$L=\ell (Y,A)$};

\end{tikzpicture}
    \caption{Common Parent}\label{fig:bayesian:standard}
    \end{subfigure}
    \hfill
    \begin{subfigure}[b]{0.3\linewidth}
    \centering
    \tikzset{every picture/.style={line width=0.75pt}} 

\begin{tikzpicture}[x=0.75pt,y=0.75pt,yscale=-0.6,xscale=0.6]

\draw   (113,72) .. controls (113,63.72) and (126.43,57) .. (143,57) .. controls (159.57,57) and (173,63.72) .. (173,72) .. controls (173,80.28) and (159.57,87) .. (143,87) .. controls (126.43,87) and (113,80.28) .. (113,72) -- cycle ;
\draw   (154,121) .. controls (154,112.72) and (167.43,106) .. (184,106) .. controls (200.57,106) and (214,112.72) .. (214,121) .. controls (214,129.28) and (200.57,136) .. (184,136) .. controls (167.43,136) and (154,129.28) .. (154,121) -- cycle ;
\draw   (56,121) .. controls (56,112.72) and (69.43,106) .. (86,106) .. controls (102.57,106) and (116,112.72) .. (116,121) .. controls (116,129.28) and (102.57,136) .. (86,136) .. controls (69.43,136) and (56,129.28) .. (56,121) -- cycle ;
\draw   (43,173) .. controls (43,164.72) and (63.15,158) .. (88,158) .. controls (112.85,158) and (133,164.72) .. (133,173) .. controls (133,181.28) and (112.85,188) .. (88,188) .. controls (63.15,188) and (43,181.28) .. (43,173) -- cycle ;
\draw   (99,212.5) .. controls (99,202.28) and (119.15,194) .. (144,194) .. controls (168.85,194) and (189,202.28) .. (189,212.5) .. controls (189,222.72) and (168.85,231) .. (144,231) .. controls (119.15,231) and (99,222.72) .. (99,212.5) -- cycle ;
\draw    (117,80) -- (91.47,103.64) ;
\draw [shift={(90,105)}, rotate = 317.2] [color={rgb, 255:red, 0; green, 0; blue, 0 }  ][line width=0.75]    (10.93,-3.29) .. controls (6.95,-1.4) and (3.31,-0.3) .. (0,0) .. controls (3.31,0.3) and (6.95,1.4) .. (10.93,3.29)   ;
\draw    (170,80) -- (188.75,103.44) ;
\draw [shift={(190,105)}, rotate = 231.34] [color={rgb, 255:red, 0; green, 0; blue, 0 }  ][line width=0.75]    (10.93,-3.29) .. controls (6.95,-1.4) and (3.31,-0.3) .. (0,0) .. controls (3.31,0.3) and (6.95,1.4) .. (10.93,3.29)   ;
\draw    (180,135) -- (160.64,192.11) ;
\draw [shift={(160,194)}, rotate = 288.73] [color={rgb, 255:red, 0; green, 0; blue, 0 }  ][line width=0.75]    (10.93,-3.29) .. controls (6.95,-1.4) and (3.31,-0.3) .. (0,0) .. controls (3.31,0.3) and (6.95,1.4) .. (10.93,3.29)   ;
\draw    (86,138) -- (86,156) ;
\draw [shift={(86,158)}, rotate = 270] [color={rgb, 255:red, 0; green, 0; blue, 0 }  ][line width=0.75]    (10.93,-3.29) .. controls (6.95,-1.4) and (3.31,-0.3) .. (0,0) .. controls (3.31,0.3) and (6.95,1.4) .. (10.93,3.29)   ;
\draw    (87,188) -- (103.42,200.77) ;
\draw [shift={(105,202)}, rotate = 217.87] [color={rgb, 255:red, 0; green, 0; blue, 0 }  ][line width=0.75]    (10.93,-3.29) .. controls (6.95,-1.4) and (3.31,-0.3) .. (0,0) .. controls (3.31,0.3) and (6.95,1.4) .. (10.93,3.29)   ;

\draw    (117,120) -- (154.5,120) ;
\draw [shift={(117,120)}, rotate = 0] [color={rgb, 255:red, 0; green, 0; blue, 0 }  ][line width=0.75]    (10.93,-3.29) .. controls (6.95,-1.4) and (3.31,-0.3) .. (0,0) .. controls (3.31,0.3) and (6.95,1.4) .. (10.93,3.29)   ;

\draw (143,72) node [anchor=center][inner sep=0.75pt]    {\tiny$Z \in \mathcal{Z} $};
\draw (86,121) node [anchor=center][inner sep=0.75pt]    {\tiny$X\in \mathcal{X}$};
\draw (184, 121) node [anchor=center][inner sep=0.75pt]    {\tiny$Y\in \mathcal{Y}$};
\draw (88, 173) node [anchor=center][inner sep=0.75pt]    {\tiny$A=m( X)$};
\draw (144,212.5) node [anchor=center][inner sep=0.75pt]    {\tiny$L=\ell (Y,A)$};

\end{tikzpicture}
    \caption{Dependent $X$}\label{fig:bayesian:measurement}
    \end{subfigure}
    \hfill
    \begin{subfigure}[b]{0.3\linewidth}
    \centering
    \tikzset{every picture/.style={line width=0.75pt}} 

\begin{tikzpicture}[x=0.75pt,y=0.75pt,yscale=-0.6,xscale=0.6]

\draw   (113,72) .. controls (113,63.72) and (126.43,57) .. (143,57) .. controls (159.57,57) and (173,63.72) .. (173,72) .. controls (173,80.28) and (159.57,87) .. (143,87) .. controls (126.43,87) and (113,80.28) .. (113,72) -- cycle ;
\draw   (154,121) .. controls (154,112.72) and (167.43,106) .. (184,106) .. controls (200.57,106) and (214,112.72) .. (214,121) .. controls (214,129.28) and (200.57,136) .. (184,136) .. controls (167.43,136) and (154,129.28) .. (154,121) -- cycle ;
\draw   (56,121) .. controls (56,112.72) and (69.43,106) .. (86,106) .. controls (102.57,106) and (116,112.72) .. (116,121) .. controls (116,129.28) and (102.57,136) .. (86,136) .. controls (69.43,136) and (56,129.28) .. (56,121) -- cycle ;
\draw   (43,173) .. controls (43,164.72) and (63.15,158) .. (88,158) .. controls (112.85,158) and (133,164.72) .. (133,173) .. controls (133,181.28) and (112.85,188) .. (88,188) .. controls (63.15,188) and (43,181.28) .. (43,173) -- cycle ;
\draw   (99,212.5) .. controls (99,202.28) and (119.15,194) .. (144,194) .. controls (168.85,194) and (189,202.28) .. (189,212.5) .. controls (189,222.72) and (168.85,231) .. (144,231) .. controls (119.15,231) and (99,222.72) .. (99,212.5) -- cycle ;
\draw    (117,80) -- (91.47,103.64) ;
\draw [shift={(90,105)}, rotate = 317.2] [color={rgb, 255:red, 0; green, 0; blue, 0 }  ][line width=0.75]    (10.93,-3.29) .. controls (6.95,-1.4) and (3.31,-0.3) .. (0,0) .. controls (3.31,0.3) and (6.95,1.4) .. (10.93,3.29)   ;
\draw    (170,80) -- (188.75,103.44) ;
\draw [shift={(190,105)}, rotate = 231.34] [color={rgb, 255:red, 0; green, 0; blue, 0 }  ][line width=0.75]    (10.93,-3.29) .. controls (6.95,-1.4) and (3.31,-0.3) .. (0,0) .. controls (3.31,0.3) and (6.95,1.4) .. (10.93,3.29)   ;
\draw    (180,135) -- (160.64,192.11) ;
\draw [shift={(160,194)}, rotate = 288.73] [color={rgb, 255:red, 0; green, 0; blue, 0 }  ][line width=0.75]    (10.93,-3.29) .. controls (6.95,-1.4) and (3.31,-0.3) .. (0,0) .. controls (3.31,0.3) and (6.95,1.4) .. (10.93,3.29)   ;
\draw    (86,138) -- (86,156) ;
\draw [shift={(86,158)}, rotate = 270] [color={rgb, 255:red, 0; green, 0; blue, 0 }  ][line width=0.75]    (10.93,-3.29) .. controls (6.95,-1.4) and (3.31,-0.3) .. (0,0) .. controls (3.31,0.3) and (6.95,1.4) .. (10.93,3.29)   ;
\draw    (87,188) -- (103.42,200.77) ;
\draw [shift={(105,202)}, rotate = 217.87] [color={rgb, 255:red, 0; green, 0; blue, 0 }  ][line width=0.75]    (10.93,-3.29) .. controls (6.95,-1.4) and (3.31,-0.3) .. (0,0) .. controls (3.31,0.3) and (6.95,1.4) .. (10.93,3.29)   ;

\draw    (116,120) -- (152,120) ;
\draw [shift={(153,120)}, rotate = 180] [color={rgb, 255:red, 0; green, 0; blue, 0 }  ][line width=0.75]    (10.93,-3.29) .. controls (6.95,-1.4) and (3.31,-0.3) .. (0,0) .. controls (3.31,0.3) and (6.95,1.4) .. (10.93,3.29)   ;

\draw (143,72) node [anchor=center][inner sep=0.75pt]    {\tiny$Z \in \mathcal{Z} $};
\draw (86,121) node [anchor=center][inner sep=0.75pt]    {\tiny$X\in \mathcal{X}$};
\draw (184, 121) node [anchor=center][inner sep=0.75pt]    {\tiny$Y\in \mathcal{Y}$};
\draw (88, 173) node [anchor=center][inner sep=0.75pt]    {\tiny$A=m( X)$};
\draw (144,212.5) node [anchor=center][inner sep=0.75pt]    {\tiny$L=\ell (Y,A)$};

\end{tikzpicture}
    \caption{Influential $X$}\label{fig:bayesian:markov}
    \end{subfigure}
    \hfill
    \caption{Bayesian networks visualizing the possible relationships between $Y$ and $X$. The decision map is denoted by $m$, the loss function by $\ell$. Both the decision $A$ and loss $L$ are random variables because they are deterministic mappings of random variables.}
    \label{fig:causal-model}
\end{figure}
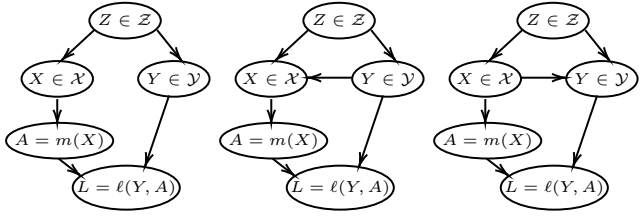
This problem minimizes the expected value of a differentiable, bounded loss function $\ell:\mathcal{Y}\times \mathcal{A}\to\mathbb{R}$, which depends on the uncertain problem parameter~$Y$ and the decision~$a\in\mathcal A$, conditional on the observation~$X$. The objective function of~\eqref{eq:SP} is commonly referred to as the Bayesian posterior loss \citep[Definition~8]{berger2013bayesian_opt}. In an energy dispatch problem, for example, $Y\in\mathcal Y=\mathbb R$ denotes the uncertain wind energy production, $a\in\mathcal{A}=\times_{j=1}^J[0, \Bar{a}_j]$ represents the energy outputs of~$J$ conventional generators with respective capacities $\Bar{a}_j$, $j=1,\ldots,J$, and 
\[
    \textstyle \ell(Y,a)=c^\top a + p\cdot\max\left\{d-Y-\sum_{j=1}^J a_j,0\right\}
\]
captures the production cost $c^\top a$ of the generators and a penalty for unmet demand. Here, $d$ stands for the total demand, $Y+\sum_{j=1}^J a_j$ represents the total energy production, and $p>0$ is a prescribed penalty parameter. In this example, the observation~$X$ can have different meanings.\\
\textbf{Figure~\ref{fig:bayesian:standard}:} Assume that the confounder~$Z\sim \mathcal N(\mu_Z,\sigma_Z^2)$ represents the (unknown) mean of $Y\sim\mathcal N(Z,\sigma^2_Y)$. In this case $X=[\widehat Y_1,\ldots,\widehat Y_N]$ may represent a collection of~$N$ historical wind production levels. If $\widehat Y_n\sim\mathcal N(Z,\sigma_Y^2)$, $n=1,\ldots,N$, are i.i.d.\ samples from~$\mathbb P_{Y|Z}$, then one can show that the posterior distribution of~$Y$ given~$X$ is 
\[
    \textstyle \mathcal{N}\left( \left(\frac{\mu_Z}{\sigma_Z^2}+\!\frac{\sum_{n=1}^N \hat Y_n}{\sigma_Y^2}\right)\!/(\frac{1}{\sigma_Z^2}+\frac{N}{\sigma_Y^2}),\sigma_Y^2+(\frac{1}{\sigma_Z^2}+\frac{N}{\sigma_Y^2})^{-1}\right).
\]
\textbf{Figure~\ref{fig:bayesian:measurement}:} The wind power~$Y\sim\mathcal N(\mu_Y,\sigma^2_Y)$ may be indirectly observable through the output $X\sim \mathcal N(Y,\sigma^2_X)$ of a noisy power meter. Hence, the posterior distribution of~$Y$ given~$X$ is $\mathcal N(\mu_Y+\sigma^2_Y(\sigma^2_Y+\sigma^2_X)^{-1}(X-\mu_Y), \sigma_Y^2-\sigma_Y^4(\sigma_Y^2+\sigma_X^2)^{-1})$, and $Y$ becomes observable as~$\sigma^2_X$ drops.\\
\textbf{Figure~\ref{fig:bayesian:markov}:} The observation~$X$ could represent the wind speed at a nearby location, which has a causal impact on~$Y$. Exploiting such contextual information leads to better decisions. Assuming normality, the posterior distribution of~$Y$ given~$X$ adopts a similar form as above. Details are omitted.

By the interchangeability principle~\citep[Theorem~14.60]{rockafellar2009variational}, problem~\eqref{eq:SP} is equivalent to
\begin{equation}
\label{eq:SP-static}
    \min_{m\in \mathcal{M}} \EE{}[\ell(Y, m(X)],
\end{equation}
where $\mathcal M$ denotes the space of all measurable decision maps~$m$ from~$\mathcal{X}$ to~$\mathcal{A}$. Specifically, $m^\star$ solves~\eqref{eq:SP-static} if and only if $a^\star=m^\star(X)$ solves~\eqref{eq:SP} for every realization of~$X$. 

We assume from now on that the joint distribution of~$X$, $Y$ and~$Z$ is unknown. In the special situation displayed in Figure~\ref{fig:bayesian:standard}, a decision map feasible in~\eqref{eq:SP-static} can be computed from the observation~$X$ alone if $X=[\widehat Y_1,\ldots,\widehat Y_N]$ consists of sufficiently many i.i.d.\ samples from~$\mathbb P_{Y|Z}$. Indeed, $m(X)$ can be defined as a solution of the ERM problem
\begin{equation}
    \label{eq:erm}
    \min_{a\in \mathcal{A}}\textstyle \frac{1}{N}\sum_{i=1}^N\ell(\widehat Y_i,a).    
\end{equation}
Alternatively, $m(X)$ can be defined as a solution of the DRO problem \cite{delage2010distributionally,wiesemann2014distributionally} 
\begin{equation}
    \label{eq:dro}
    \min_{a\in\mathcal{A}}\max_{\mathbb Q\in\mathcal{U}(X)}\textstyle  \int_{\mathcal Y}\ell(y,a) \, {\rm d}\mathbb Q(y),
\end{equation}
where $\mathcal{U}(X)$ represents an ambiguity set, that is, a family of distributions~$\mathbb Q$ of~$Y$ that are sufficiently likely to have generated the samples~$\widehat Y_1,\ldots,\widehat Y_N$. Popular choices of the ambiguity set $\mathcal{U}(X)$ are surveyed in~\cite{ref:Rahimian-19}. However, both ERM and DRO do not readily extend to the situations depicted in Figures~\ref{fig:bayesian:measurement} and~\ref{fig:bayesian:markov}.

In the remainder of the paper we assume that we have access to i.i.d.\ training samples $\{(X_k,Y_k)\}_{k=1}^K$ (recall that the confounder~$Z$ is {\em not} observable). A feasible decision map can then be obtained via the \textit{predict-then-optimize} approach, which trains a regression model to predict~$Y$ from~$X$ and defines $m(X)$ as an action that minimizes the loss of the prediction, see, e.g., \cite{mivsic2020data}. The resulting decisions are tailored to the point prediction at hand but may incur high losses when $Y$ deviates from its prediction. In addition, the regression model used for the prediction is usually agnostic of the downstream optimization model~\eqref{eq:SP}; a notable exception being \cite{ref:Elmachtoub-22}. Instead of predicting~$Y$ from~$X$, one can use machine learning methods to predict the conditional distribution~$\mathbb P_{Y|X}$ of~$Y$ given~$X$ and define~$m(X)$ as a solution of~\eqref{eq:SP} under the estimated distribution \cite{bertsimas2020predictive}. However, the methods that are used for estimating~$\mathbb P_{Y|X}$ are again agnostic of the downstream optimization model~\eqref{eq:SP}. 

All approaches reviewed so far have the shortcoming that evaluating~$m(X)$ necessitates the solution of a potentially large optimization problem. In contrast, \textit{end-to-end learning} \cite{donti2017end2end_constr,fu2018end, cvxpylayers2019, uysal2021end, zhang2021universal} trains a parametric decision map~$m$ that is near-optimal in~\eqref{eq:SP-static} {\em without} the detour of first estimating~$Y$ or its distribution. This is achieved by applying stochastic gradient descent (SGD) directly to~\eqref{eq:SP-static}. Thus, end-to-end learning {\em (i)} avoids the artificial separation of estimation and optimization characteristic for competing methods, {\em (ii)} enjoys high scalability because the decision map is trained using SGD and can be evaluated efficiently without solving any optimization model, and {\em (iii)} can even handle unstructured observations~$X$ such as text or images.

In the following, we describe several key aspects of an end-to-end learning model.
In Section~\ref{sec:model-architecture} we discuss neural network architectures that lend themselves to representing decision maps. In Section~\ref{sec:training-process} we then review a popular SGD-based training method and prove that the resulting decision map approximately minimizes the Bayesian posterior loss $\EE{}[\ell(Y,a)|X]$ under the prior $\mathbb P_{(X,Y)}$ and the observation $X$. 

\begin{remark}[Generalized Data Sets]
\label{rem:general-training-sets}
All results of this paper extend to training sets of the form $\{(X_k,\widehat{\PP{}}_k{})\}_{k=1}^K$, where $\widehat{\PP{}}_k{}$ represents an unbiased estimator for the conditional distribution $\PP{Y|X_k}$ in the sense that
\[
    \textstyle \EE{}[\int_{\mathcal Y}\ell(y,a)\,{\rm d}\widehat{\mathbb P}_k(y)|X_k]=\EE{}[\ell(Y,a)|X_k]~ \forall k=1,\ldots,K.
\]
Note that if $(X_k,Y_k)$ is sampled from~$\mathbb P_{(X,Y)}$, for example, then the Dirac distribution $\widehat{\mathbb P}_k=\delta_{Y_k}$ constitutes an unbiased estimator for~$\PP{Y|X_k}$. Hence, the dataset $\{(X_k,\widehat{\PP{}}_k{})\}_{k=1}^K$ strictly generalizes the standard dataset $\{(X_k,Y_k)\}_{k=1}^K$.
\end{remark}

\section{Model Architecture}\label{sec:model-architecture}
A fundamental design choice for end-to-end learning models is the architecture of the decision map $m:\mathcal{X}\rightarrow \mathcal{A}$. Throughout this section, we represent~$m$ as a neural network obtained by combining a feature extractor $f:\mathcal{X}\to \mathcal{R}$ with a prescriptor $p:\mathcal{R}\to \mathcal{A}$, where $\mathcal{R}$ denotes the feature space. The complete network is thus given by $m=p \circ f$. Below we review possible choices for both $f$ and $p$.

\subsection{Feature Extractor}
As in classical machine learning tasks, the choice of the architecture for the feature extractor is mainly informed by the data format and by the desired symmetry and invariance properties. Thus, the feature extractor may include linear layers~\cite{rumelhart1986learning}, convolutional layers~\cite{zhang1988shift, lecun1989lenet}, attention layers~\cite{vaswani2017attention} and recurrent layers~\cite{elman1990finding, hochreiter1997long, cho2014properties} etc., combined with activation functions and regularization layers.

\subsection{Prescriptor}
We propose three different architectures for the prescriptor.\\
\textbf{(A) Multi-Layer Perceptron (MLP).} 
 \citet{uysal2021end} use classical neural network layers for the prescriptor. In this case the decision map reduces to a MLP, which enjoys great expressive power thanks to various universal approximation theorems; see \cite{cybenko1989approximation,ref:Barron-91, NIPS2017_32cbf687,ref:Delalleau-11} and references therein. \\
\textbf{(B) Constraint-Aware Layers.} 
The stochastic program~\eqref{eq:SP} often involves constraints that ensure compliance with physical or regulatory requirements (such as maximum driving voltage constraints in robotics or short-sales constraints in portfolio selection). To ensure that the decision map satisfies all constraints, we use an output layer that maps any input to the corresponding feasible set. This can be achieved in two different ways. Sometimes, one can manually design activation functions tailored to the constraint set at hand~\cite{zhang2021universal}. For example, the smooth softmax function maps any input into the probability simplex. However, more complicated constraint sets require a more systematic treatment. If the constraint set is closed and convex, for example, then one can construct an output layer that projects any input onto the feasible set. Unfortunately, this approach suffers from a gradient projection problem outlined in Section~\ref{sec:additional-considerations}.\\
\textbf{(C) Optimization Layers.} 
In view of~\eqref{eq:SP}, it is natural to define the prescriptor as the `argmin' map of a parametric optimization model. Specifically, the prescriptor may output the solution of a {\em deterministic} optimization model that minimizes the loss at a point estimate of~$Y$. Alternatively, it may output the solution of a {\em stochastic} optimization model that minimizes the {\em expected} loss under an estimator for the conditional distribution~$\mathbb P_{Y|X}$. In both cases, the Jacobian of the prescriptor with respect to the estimator, which is an essential ingredient for SGD-type methods, can quite generally be derived from the problem's KKT conditions \cite{donti2017end2end_constr, cvxpylayers2019, uysal2021end}. A key advantage of optimization layers is their ability to capture prior structural information. They are also highly interpretable because the features can be viewed as predictions of~$Y$ or~$\mathbb P_{Y|X}$. Like constraint-aware layers, however, optimization layers suffer from a gradient projection problem that is easy to overlook; see Section~\ref{sec:additional-considerations}.

\subsection{Approximation Capabilities}
\label{ssec:approximation:capabilities}
It is well known that MLPs can uniformly approximate any continuous function even if they only have one hidden layer \cite{cybenko1989approximation, NIPS2017_32cbf687}. We will now show that the decision maps considered in this paper inherit the universal approximation capabilities from the feature extractor.
\begin{proposition}[Universal Approximation of $m$]\label{prop:universal_approximation}
Assume that $m=p\circ f$ combines a feature extractor $f:\mathcal{X}\to \mathcal{R}$ with a prescriptor $p:\mathcal{R}\to \mathcal{A}$ and that~$p$ is $L_p$-Lipschitz continuous. Then, the following hold.
\vspace{-4mm}
\begin{enumerate}
    \item[(i)] If there exists a neural network $f_w:\mathcal{X}\rightarrow\mathcal{R}$ with $\sup_{x\in\mathcal{X}}\|f(x)-f_w(x)\|\leq\varepsilon$, then $m_w=p\circ f_w$ satisfies $\sup_{x\in\mathcal{X}}\|m(x)-m_w(x)\|\leq L_p\varepsilon$.
    \vspace{-2mm}
    \item[(ii)] If there exists a neural network $f_w:\mathcal{X}\rightarrow\mathcal{R}$ with $\mathbb E[\|f(X)-f_w(X)\|^q] \leq \varepsilon$ for some $q\geq 1$, then $m_w=p\circ f_w$ satisfies $\mathbb E[\|m(X)-m_w(X)\|^q] \leq L_p^q\varepsilon$.
\end{enumerate}
\end{proposition}

The following corollary shows that neural networks can not only be used to approximate $m$ but also its expected loss. 

\begin{corollary}[Universal Approximation of the Loss]
\label{cor:universal-approx-loss}
Assume that $m=p\circ f$ combines a feature extractor $f:\mathcal{X}\to \mathcal{R}$ with a prescriptor $p:\mathcal{R}\to \mathcal{A}$, that $f$ is continuous and that~$p$ is $L_p$-Lipschitz continuous. If $\mathcal{X}$ is bounded and the loss function $\ell(y,a)$ is $L_\ell$-Lipschitz in $a$ uniformly across all $y\in\mathcal Y$, then for every $\varepsilon>0$ there exists a neural network $f_w:\mathcal X\to\mathcal R$ with sigmoid activation that satisfies 
$$
\left|\EE{}[\ell(Y,m(X)]-\EE{}[\ell(Y,m_w(X))]\right|\leq\varepsilon.
$$
\end{corollary}

Corollary~\ref{cor:universal-approx-loss} implies that, for every $\varepsilon>0$, there exists a neural network $f_w$ and a decision map $m_w=p\circ f_w$ with
$$
\EE{}[\ell(Y,m(X)]\leq\EE{}[\ell(Y,m_w(X))]+\varepsilon.
$$
In other words, there exists a neural network-based map $m_w$ whose the expected loss exceeds that of $m$ at most by~$\varepsilon$.

\section{Training Process and Loss Function}\label{sec:training-process}
The decision map~$m^\star$ that solves~\eqref{eq:SP-static} under the unknown true distribution of $X$ and $Y$ is inaccessible. However, we can train a neural network~$m_w$ parametrized by~$w\in\mathbb R^d$ that approximates~$m^\star$. In the following we review a popular SGD-type algorithm for training~$m_w$. While the intimate relation between this widely used algorithm and problems~\eqref{eq:SP} and~\eqref{eq:SP-static} has not yet been investigated, we prove that~$m_w$ maps any observation~$X$ to an approximate posterior Bayes action corresponding to~$X$. When $X$ represents a collection of i.i.d.\ samples from $\mathbb P_{Y|Z}$, finally, we outline alternative training methods under which~$m_w$ maps~$X$ to an approximate minimizer of an ERM or a DRO problem akin to~\eqref{eq:SP}.

\subsection{SGD-Type Algorithm}
End-to-end learning problems are commonly addressed with Algorithm~\ref{algo:training} below \cite{uysal2021end,zhang2021universal}.

\begin{algorithm}
\caption{End-to-End Learning}
 \label{algo:training}
\begin{algorithmic}
 \FOR{$k\gets1,\hdots,K$}
    \STATE $g_k \gets \nabla_w \ell(Y_k,m_w(X_k))|_{w=w_{k-1}}$
    \STATE $w_k \gets w_{k-1} - \eta_k g_k$
 \ENDFOR
 \end{algorithmic}
\end{algorithm}

Note that $\nabla_w \ell(Y_k,m_w(X_k))$ constitutes an unbiased estimator for $\nabla_w\mathbb E[\ell(Y,m_w(X))]$. Thus, Algorithm~\ref{algo:training} can be viewed as an SGD method for training the decision map~$m_w$. Note also that the step size $\eta_k>0$ may depend on time.

\begin{remark}[Generalized Data Sets]
\label{rem:general-training-sets2}
Given a generalized dataset $\{(X_k,\widehat{\PP{}}_k{})\}_{k=1}^K$
as described in Remark~\ref{rem:general-training-sets}, we can use $\nabla_w\int_{\mathcal Y}\ell(y,m_w(X_k)){\rm d}\widehat{\PP{}}_k{}(y) |_{w=w_{k-1}}$ as an unbiased gradient estimator in Algorithm~\ref{algo:training} (if it is well-defined).
\end{remark}

\subsection{Bayesian Interpretation of Algorithm~\ref{algo:training}}
We now show that if the decision map $m_w$ is trained via Algorithm~\ref{algo:training}, then $m_w(X)$ constitutes an approximate posterior Bayes action. That is, it approximately solves problem~\eqref{eq:SP}, which minimizes the Bayesian posterior loss.
The Bayesian posterior $\mathbb P_{Y|X}$ reflects the information available from a given prior $\mathbb P_{(X,Y)}$ and an observation~$X$. It is widely used in various decision-making problems~\cite{kalmanfilter, stengel1994optimal, pezeshk2003bayesian,long2010active}.

\subsubsection{Theoretical Analysis of Algorithm~\ref{algo:training}} \label{ssec:theorem1}
Even though Algorithm~\ref{algo:training} is commonly used and conceptually simple, the training loss it minimizes has not yet been investigated. We now analyze Algorithm~\ref{algo:training} theoretically. The following standard assumption is required for convergence results of all methods under consideration.
\begin{assumption}[Smoothness]\label{ass:smoothness}
The loss function $\ell(y,a)$ is smooth in $a$ for all $y\in\mathcal{Y}$, the decision map $ m_w(x)$ is smooth in $w$ for all $x\in\mathcal{X}$, and their gradients are bounded.
\end{assumption}
Replacing the space of all measurable decision maps by the set of all neural network-based decision maps $m_w$ with parameter $w\in\mathbb R^d$ yields the following approximation of~\eqref{eq:SP-static}.
\begin{equation}
\label{eq:bayesian-interpretation}
    \min_{w\in\mathbb R^d}  \EE{}[\ell(Y, m_w(X))]
\end{equation}
From now on we use $\varphi(w)=\EE{}[\ell(Y, m_w(X))]$ as a shorthand for the objective function of~\eqref{eq:bayesian-interpretation}.
Next, we prove that Algorithm~\ref{algo:training} converges to a stationary point of problem~\eqref{eq:bayesian-interpretation}.
\begin{theorem}[Bayesian Interpretation of Algorithm~\ref{algo:training}] \label{thm:equivalence:algo:1}
Algorithm~\ref{algo:training} solves problem~\eqref{eq:bayesian-interpretation} in the following sense.
\begin{enumerate}
\vspace{-4mm}
    \item[(i)] The vector $g_k$ computed in Algorithm~\ref{algo:training} constitutes an unbiased stochastic gradient for $\varphi(w)$ at $w=w_{k-1}$.
    \item[(ii)] If Assumption~\ref{ass:smoothness} holds and Algorithm~\ref{algo:training} uses step sizes $\eta_k\propto 1/\sqrt{k}$, then the random iterate $\widehat{w}_K$ sampled from $\{w_k\}_{k=1}^K$ with respective probabilities $\{p_k\}_{k=1}^K$, $p_k\propto\eta_k^{-1}$, satisfies $\mathbb{E}[\|\nabla \varphi(\widehat{w}_K)\|_2^2]=O(1/\sqrt{K})$.
\end{enumerate}
\end{theorem}


Theorem~\ref{thm:equivalence:algo:1} implies that if Assumption~\ref{ass:smoothness} holds, then the gradient of $\varphi(w)$ at the random iterate $\widehat{w}_K$ generated by Algorithm~\ref{algo:training} is small in expectation and---by virtue of Markov's inequality---also small with high probability. Thus, $\widehat w_K$ converges in probability to a stationary point~$w^\star$ of problem~\eqref{eq:bayesian-interpretation} as $K$ tends to infinity. Throughout the subsequent discussion we assume that $w^\star$ is in fact a minimizer of~\eqref{eq:bayesian-interpretation}.
Theorem~\ref{thm:equivalence:algo:1} then suggests that the neural network~$m_{w^\star}$ maps any observation~$X$ to an {\em approximation} of the posterior Bayes action corresponding to~$X$. To see this, recall that any minimizer~$m^\star$ of~\eqref{eq:SP-static} maps any observation~$X$ to the
{\em exact} posterior Bayes action~$m^\star(X)$, which solves the original stochastic optimization problem~\eqref{eq:SP}. If the family of neural networks $\mathcal M_w=\{m_w:w\in\mathbb R^d\}$ is rich enough to contain a minimizer~$m^\star$ of~\eqref{eq:SP-static}, then~\eqref{eq:SP-static} and~\eqref{eq:bayesian-interpretation} are clearly equivalent. A sufficient condition for this is that for every function~$m_0\in\mathcal M_w$, every measurable set~$\mathcal B\subseteq \mathcal X$ and every bounded measurable function~$m_1\in\mathcal M$, the function $m:\mathcal X\to \mathcal A$ defined through $m(x)=m_0(x)$ if $x\in\mathcal B$ and $m(x)=m_1(x)$ if $x\notin\mathcal B$ is also a member of~$\mathcal M_w$; see~\citep[Theorem~14.60]{rockafellar2009variational}. The next corollary describes a situation in which this condition holds.

\begin{proposition}[Finite Observation Spaces]
\label{prop:bayesian-discrete}
Assume that $\mathcal X=\{x_1,\ldots, x_n\}$ is finite and that the family of neural networks $\mathcal M_w=\{m_w:w\in\mathbb R^d\}$ is rich enough such that for every $a\in\mathbb R^n$ there exists $w\in\mathbb R^d$ with $a_i=m_w(x_i)$ for all $i=1,\ldots,n$. Then, problems~\eqref{eq:SP-static} and~\eqref{eq:bayesian-interpretation} are equivalent.
\end{proposition}

Under the assumptions of Proposition~\ref{prop:bayesian-discrete} the family~$\mathcal{M}_w$ of neural networks is able to model a look-up table on $\mathcal{X}$. This implies that if Algorithm~\ref{algo:training} converges to the global minimizer~$w^\star$ of~\eqref{eq:bayesian-interpretation}, then $m_{w^\star}$ coincides exactly with the posterior Bayes action map. Otherwise, Algorithm~\ref{algo:training} uses SGD to approximate~$m^\star$ or, in other words, to seek an architecture-regularized approximation of~$m^\star$.

The above insights highlight the importance of choosing a training dataset that is reflective of one's prior belief about the distribution of the unobserved confounder~$Z$. Put differently, it is crucial that the dataset $\{(X_k, Y_k)\}_{k=1}^K$ is consistent with the prior data distribution $\mathbb P_{(X,Y)}$ during deployment. Indeed, a strong prior induced by the dataset may significantly bias the decisions of the model. In addition, the above insights provide some justification for augmenting the dataset with rare corner cases. Indeed, excluding corner cases amounts to setting their prior probabilities to~$0$, which may have undesirable consequences. In the context of algorithmic fairness~\cite{fair-ML19}, for instance, using a training dataset in which one demographic group is under\-represented amounts to working with a biased prior distribution and results in subpar predictions for the underrepresented group~\cite{buolamwini2018gender}.

\subsubsection{Minimum Mean-Square Estimation}\label{ssec:mean:estimation}
\begin{figure}
    \centering
    \begin{subfigure}[b]{0.45\linewidth}
    \centering
    \includegraphics[width=\linewidth]{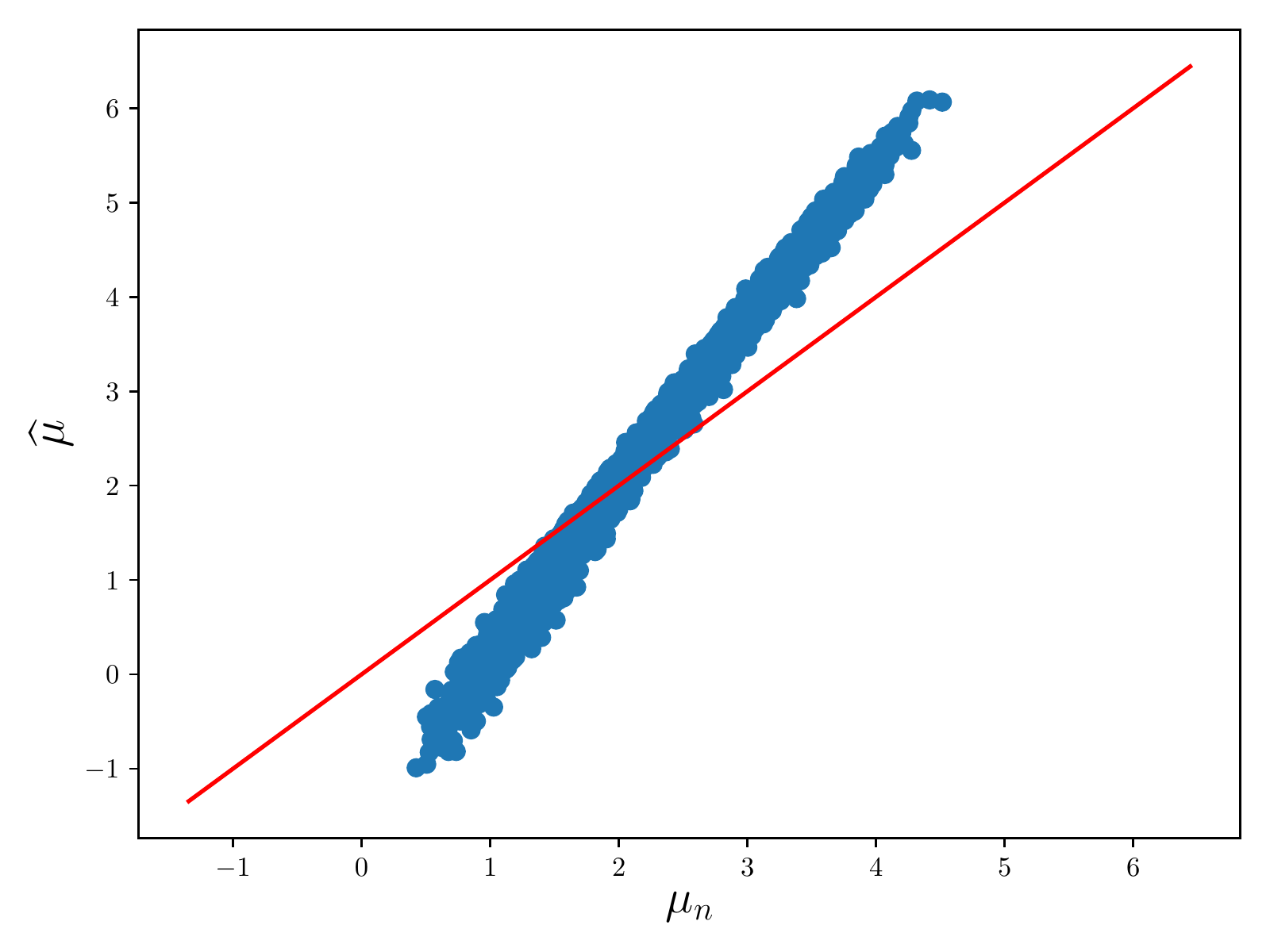}
    \caption{NN vs ERM}\label{fig:mean-estimation:1}
    \end{subfigure}
    \hfill
    \begin{subfigure}[b]{0.45\linewidth}
    \centering
    \includegraphics[width=\linewidth]{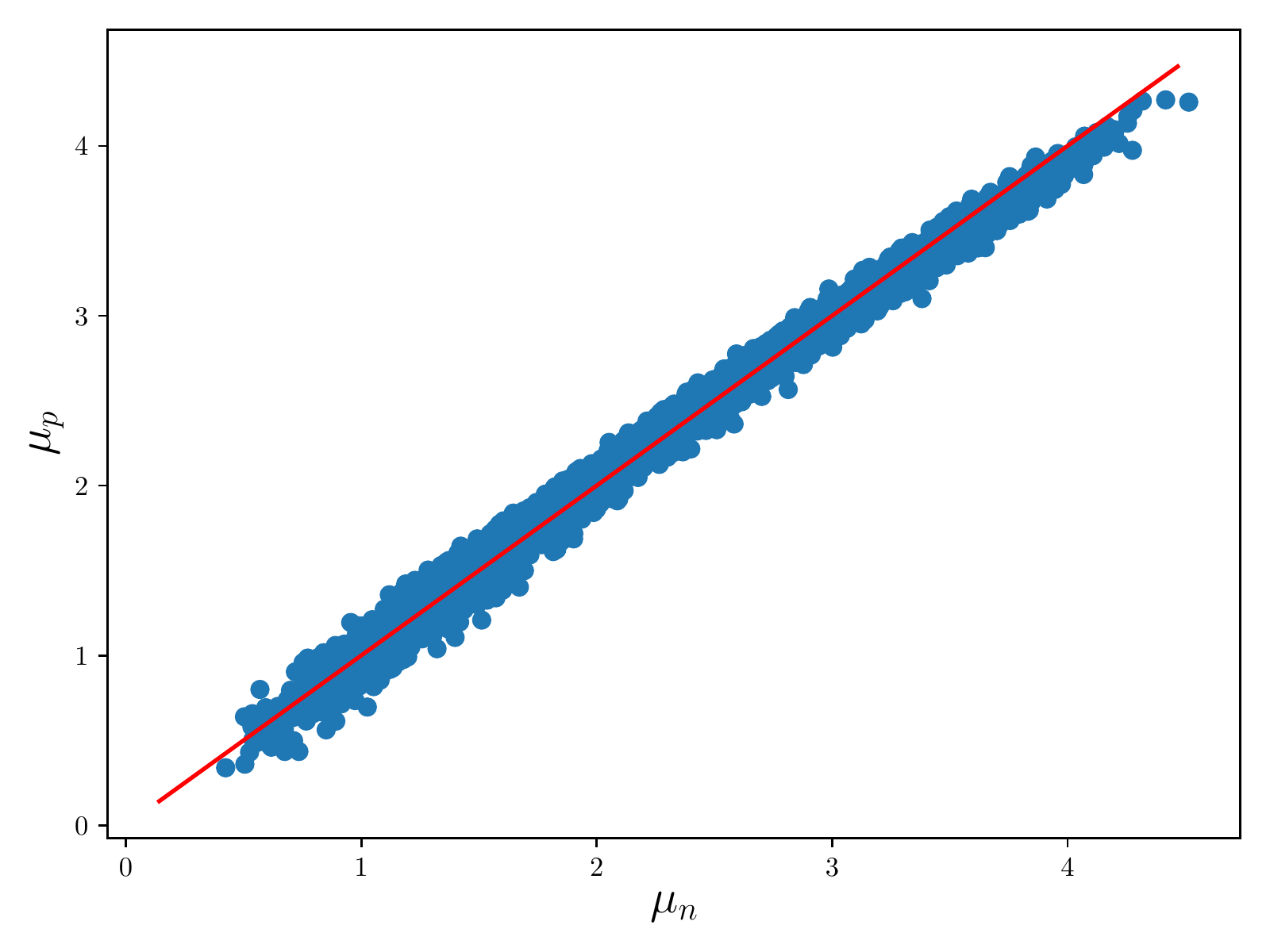}
    \caption{NN vs Posterior Mean}\label{fig:mean-estimation:2}
    \end{subfigure}
    \caption{Comparison of the approximate posterior mean $\widehat \mu_{\rm NN}$ output by Algorithm~\ref{algo:training} against $\widehat \mu_{\rm ERM}$ and $\widehat \mu_{\rm MMSE}$. As the prior~$\mathbb P_Y$ is concentrated around $\mathbb E[Y]=2$, both the posterior mean $\widehat \mu_{\rm MMSE}$ as well as its approximation $\widehat \mu_{\rm NN}$ display a bias towards $2$.}
    \label{fig:mean-estimation}
    \vspace{-0.5cm}
\end{figure}
In order to gain further intuition, consider the problem of estimating the mean~$Z$ of a Gaussian random variable $Y\sim \mathcal N(Z,4)$ based on an observation $X=[\widehat Y_1,\ldots,\widehat Y_{20}]$ of $20$ i.i.d.\ samples drawn from~$\PP{Y|Z}$. Assume that $Z\sim\mathcal{N}(2,0.25)$. The minimum mean-square estimator coincides with the solution of the stochastic optimization problem
\begin{equation}
\label{eq:mean-estimation}
    \min_{a\in \mathbb{R}} \EE{}[(Y - a)^2|X],
\end{equation}
which is readily identified as an instance of~\eqref{eq:SP}. In order to compute the minimum mean-square estimator simultaneously for all realizations of~$X$, we should solve the corresponding instance of~\eqref{eq:SP-static}. A simple calculation exploiting our distributional assumptions shows that~\eqref{eq:SP-static} is solved by~$m^\star(X) =(8+\frac{1}{4}\sum_{i=1}^{20}\widehat Y_i)/9$.
Absent any information about the joint distribution of~$X$, $Y$ and~$Z$, we have to solve the approximate problem~\eqref{eq:bayesian-interpretation} instead. Specifically, we optimize over decision maps of the form $m_w=p\circ f_w=f_w$, where the prescriptor~$p$ is the identity function, and the feature extractor~$f_w$ is a feed-forward neural network with an input layer accommodating $20$ neurons and linear activation functions, two hidden layers with $500$~neurons and ReLU-activation functions, and an output layer with $1$ neuron and a linear activation function. We solve the resulting instance of~\eqref{eq:bayesian-interpretation} using Algorithm~\ref{algo:training} with $K=5\times 10^6$ training samples $\{(X_k,Y_k)\}_{k=1}^K$ to find an approximate posterior Bayes action map~$m_{w_K}$. Given an independent test sample $X=[\widehat Y_1,\ldots,\widehat Y_{20}]$, this map outputs a neural network-based approximation~$\widehat \mu_{\rm NN}=m_{w_K}(X)$ of the exact posterior mean $\widehat\mu_{\rm MMSE}= m^\star(X)$. We compare it against the sample mean $\widehat\mu_{\rm ERM}=\frac{1}{20}\sum_{n=1}^{20} \widehat Y_i$, which minimizes the empirical risk. Figure~\ref{fig:mean-estimation} visualizes the differences between these estimators on $5{,}000$ test samples generated by sampling from $\mathbb P_{X|Z}$ for 
$5{,}000$ equidistant values of~$Z$ between~$0$ and~$5$. We observe that $\widehat\mu_{\rm MMSE}$ closely approximates the posterior mean because the scatter plot concentrates on the identity line in red. When compared to~$\widehat\mu_{\rm ERM}$, $\widehat\mu_{\rm NN}$ displays a bias towards 2 due to the strong prior.

\subsection{Alternative Decision Models}\label{sec:erm}

While minimizing the Bayesian posterior loss is uncommon in the literature, ERM and DRO are widely used if the observation $X=[\widehat Y_1,\ldots, \widehat Y_N]$ consists of $N$ i.i.d.\ samples from the unknown distribution $\mathbb P_{Y|Z}$. We now propose new neural network-based end-to-end learning algorithms that output approximate solution maps for problems~\eqref{eq:erm} and~\eqref{eq:dro}. 
An overview of our main results is provided in Table~\ref{tab:summary}.

\begin{table}[]
\centering
\caption{Overview of main results.}
\label{tab:summary}
\vspace{0.1cm}
\begin{tabular}{|c|c|c|}
 \hline
 Decision Model & Algorithm & Guarantees \\
 \hline
 Bayesian  & Algorithm~\ref{algo:training}  & Theorem~\ref{thm:equivalence:algo:1}  \\
 ERM & Algorithm~\ref{algo:training-erm}  & Theorem~\ref{thm:equivalence:algo:2}  \\
    DRO & Algorithm~\ref{algo:training-dro}  & Theorem~\ref{thm:equivalence:algo:3}  \\
\hline
\end{tabular}
\vspace{-0.5cm}
\end{table}

\subsubsection{Empirical Risk Minimization}
Given observations of the form $X = [\widehat Y_1,\ldots, \widehat Y_N]$, it is more common to solve the ERM problem~\eqref{eq:erm} instead of the Bayesian optimization problem~\eqref{eq:SP}. 
Algorithm~\ref{algo:training-erm} below learns a parametric decision map~$m_w$ with the property that $m_w(X)$ approximately solves~\eqref{eq:erm} for every realization of~$X$.


\begin{algorithm}
\caption{End-to-End Learning for ERM}
 \label{algo:training-erm}
 \begin{algorithmic}
 \FOR{$k\gets1,\hdots,K$}
    \STATE $g_k\gets\nabla_w \frac{1}{N}\sum_{n=1}^N\ell(\widehat{Y}_{k,n},m_w(X_{k}))|_{w=w_{k-1}}$
    \STATE $w_k \gets w_{k-1} - \eta_k g_k$
    \ENDFOR
 \end{algorithmic}
\end{algorithm}

Unlike Algorithm~\ref{algo:training}, which evaluates the loss in the gradient computations at a single sample~$Y_k$, which is independent of $X_k$ conditional on the unobserved confounder~$Z_k$, Algorithm~\ref{algo:training-erm} evaluates the loss at~$N$ samples $\{\widehat Y_{k,n}\}_{n=1}^N$, which are the components of~$X_k$. Thus, Algorithm~\ref{algo:training-erm} uses a training set that consists only of observations $\{X_k\}_{k=1}^K$ but not of the corresponding problem parameters $\{Y_k\}_{k=1}^K$.


Using the interchangeability principle~\citep[Theorem~14.60]{rockafellar2009variational} and the law of iterated conditional expectations, one can show that problem~\eqref{eq:erm} is equivalent to
\begin{equation}
\label{eq:ERM-static}
    \min_{m\in \mathcal{M}} \textstyle \EE{}[ \frac{1}{N}\sum_{n=1}^N\ell(\widehat Y_n, m(X))].
\end{equation}
Thus, $m^\star$ solves~\eqref{eq:ERM-static} if and only if $m^\star(X)$ solves~\eqref{eq:erm} for all realizations of $X$ (Lemma~\ref{lemma:erm-equivalence}). 
Next, we show that Algorithm~\ref{algo:training-erm} targets the following approximation of~\eqref{eq:ERM-static}.
\begin{equation}
\label{eq:bayesian-interpretation-erm}
    \min_{w\in\mathbb R^d} \textstyle \EE{}[ \frac{1}{N}\sum_{n=1}^N\ell(\widehat Y_n, m_w(X))]
\end{equation}
Below we abbreviate the objective function of~\eqref{eq:bayesian-interpretation-erm} by $\psi(w)$.


\begin{theorem}[ERM Interpretation of Algorithm~\ref{algo:training-erm}] \label{thm:equivalence:algo:2}
Algorithm~\ref{algo:training-erm} solves problem~\eqref{eq:bayesian-interpretation-erm} in the following sense.
\begin{itemize}
\vspace{-4mm}
    \item[(i)] The vector $g_k$ computed in Algorithm~\ref{algo:training-erm} constitutes an unbiased stochastic gradient for $\psi(w)$ at $w=w_{k-1}$. 
    \item[(ii)] If Assumption~\ref{ass:smoothness} holds and Algorithm~\ref{algo:training-erm} uses step sizes $\eta_k\propto 1/\sqrt{k}$, then the random iterate $\widehat{w}_K$ sampled from $\{w_k\}_{k=1}^K$ with respective probabilities $\{p_k\}_{k=1}^K$, $p_k\propto\eta_k^{-1}$, satisfies $\mathbb{E}[\|\nabla \psi(\widehat{w}_K)\|_2^2]=O(1/\sqrt{K})$.
\end{itemize}
\end{theorem}

In analogy to Section~\ref{ssec:mean:estimation}, we can use
Algorithm~\ref{algo:training-erm} to train an approximate ERM action map~$m_{w_K}$, which assigns each observation $X=[\widehat Y_1,\ldots,\widehat Y_{20}]$ an approximate sample mean $\widehat \mu_{\rm NN}=m_{w_K}(X)$. Figure~\ref{fig:mean-estimation-erm} compares $\widehat \mu_{\rm NN}$ against the exact sample mean $\widehat \mu_{\rm ERM}$ and the posterior mean $\widehat \mu_{\rm MMSE}$.
\begin{figure}
    \centering
    \begin{subfigure}[b]{0.45\linewidth}
    \centering
    \includegraphics[width=\linewidth]{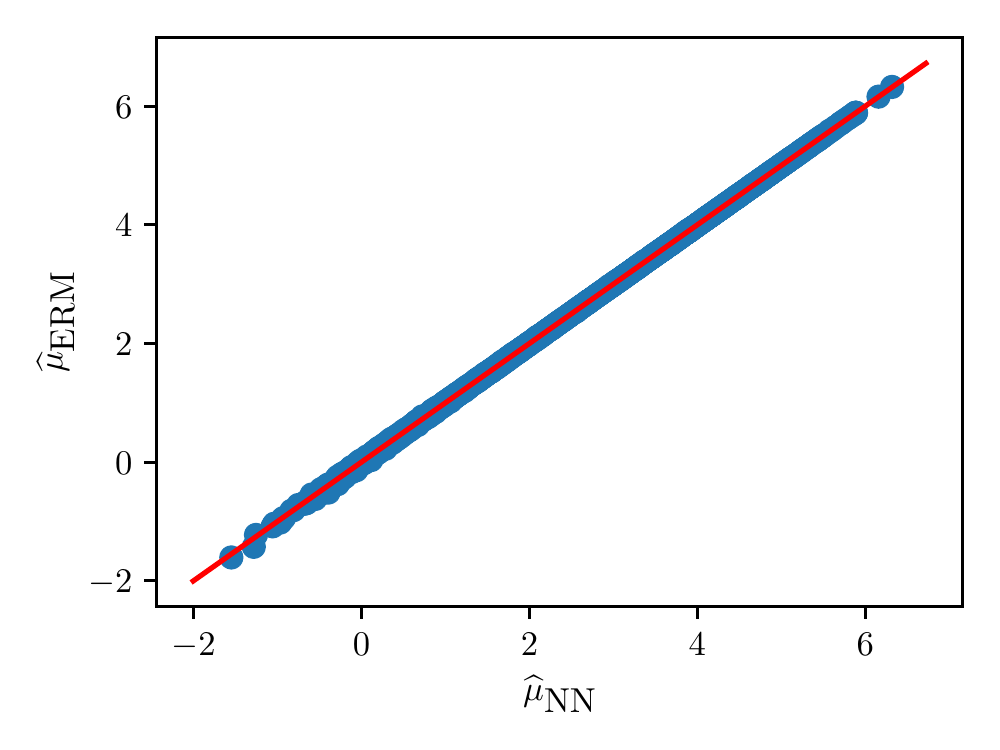}
    \caption{NN vs ERM}
    \end{subfigure}
    \hfill
    \begin{subfigure}[b]{0.45\linewidth}
    \centering
    \includegraphics[width=\linewidth]{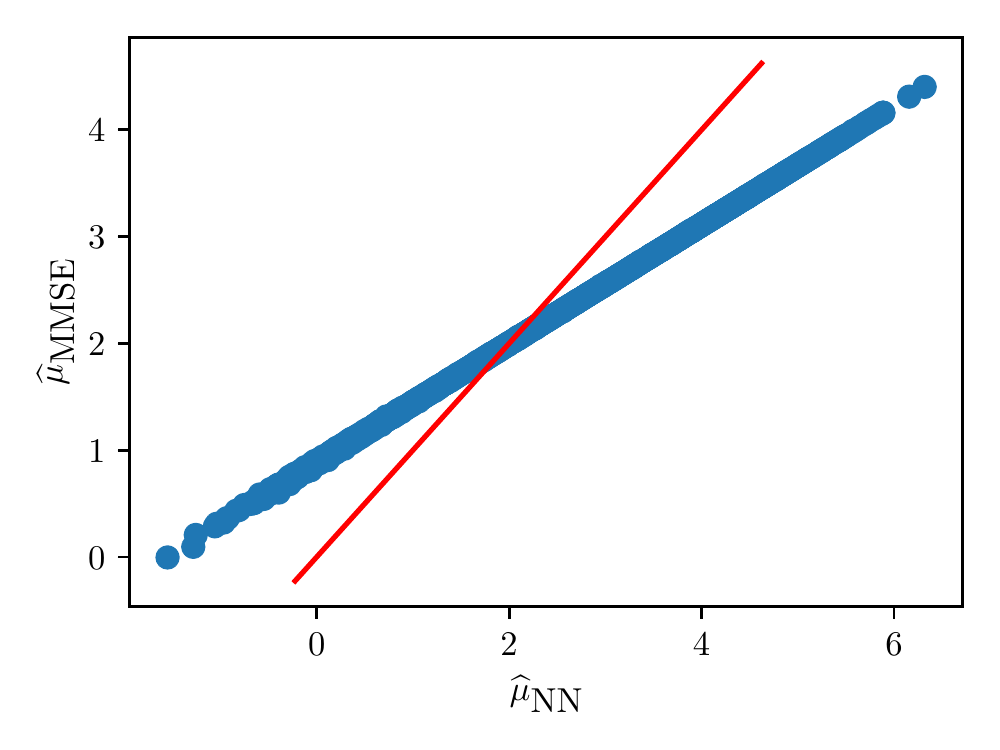}
    \caption{NN vs Posterior Mean}
    \end{subfigure}
    \caption{Comparison of the approximate sample mean $\widehat \mu_{\rm NN}$ output by Algorithm~\ref{algo:training-erm} against $\widehat \mu_{\rm ERM}$ and $\widehat \mu_{\rm MMSE}$.}
    \label{fig:mean-estimation-erm}
    \vspace{-0.5cm}
\end{figure}

\subsubsection{Distributionally Robust Optimization}
Another generalization of Algorithms~\ref{algo:training} and~\ref{algo:training-erm} is Algorithm~\ref{algo:training-dro} below, which learns a parametric decision map $m_w$ with the property that $m_w(X)$ approximately solves the DRO problem~\eqref{eq:dro} for every realization of $X$. Much like Algorithm~\ref{algo:training-erm}, Algorithm~\ref{algo:training-dro} uses a training set that consists only of observations but not of the corresponding problem parameters.


\begin{algorithm}
\caption{End-to-End Learning for DRO}
 \label{algo:training-dro}
 \begin{algorithmic}
 \FOR{$k\gets1,\hdots,K$}
 \STATE $\mathbb{Q}_k^\star \in  \arg\max_{\mathbb{Q}\in\mathcal{U}(X_k)}\int_{\mathcal Y}\ell(y, m_{w_{k-1}}(X_{k}))\,{\rm d}\mathbb{Q}(y)$
    \STATE $g_k \gets \nabla_w \int_{\mathcal Y}\ell(y,m_w(X_{k}))\,{\rm d}\mathbb{Q}_k^\star (y)|_{w=w_{k-1}}$
    \STATE $w_k \gets w_{k-1} - \eta_k g_k$
    \ENDFOR
 \end{algorithmic}
\end{algorithm}

One can show (Lemma~\ref{lemma:dro-equivalence}) that problem~\eqref{eq:dro} is equivalent~to 
\begin{equation}
\label{eq:DRO-static}
    \min_{m\in \mathcal{M}} \EE{}\left[ \max_{\mathbb Q\in\mathcal{U}(X)}\textstyle  \int_{\mathcal Y}\ell(y,m(X)) \, {\rm d}\mathbb Q(y) \right].
\end{equation}
Algorithm~\ref{algo:training-dro} then targets the following approximation of~\eqref{eq:DRO-static}.
\begin{equation}
    \label{eq:bayesian-interpretation-dro}
    \min_{w\in \mathbb R^d} \EE{}\left[ \max_{\mathbb Q\in\mathcal{U}(X)}\textstyle  \int_{\mathcal Y}\ell(y,m_w(X)) \, {\rm d}\mathbb Q(y) \right].
\end{equation}
Below we abbreviate the objective function of~\eqref{eq:bayesian-interpretation-dro} by $\chi(w)$.



\begin{theorem}[DRO Interpretation of Algorithm~\ref{algo:training-dro}] \label{thm:equivalence:algo:3}
Algorithm~\ref{algo:training-dro} solves problem~\eqref{eq:dro} in the following sense. If Assumption~\ref{ass:smoothness} holds, $\mathcal{U}(X)\neq \emptyset$ is weakly compact and the maximization problem over~$\mathbb Q$ in~\eqref{eq:bayesian-interpretation-dro} has $\mathbb P$-almost surely a unique solution, then the vector $g_k$ computed in Algorithm~\ref{algo:training-dro} is an unbiased stochastic gradient for $\chi(w)$ at $w=w_{k-1}$.
\end{theorem}

The assumption that the maximization problem over~$\mathbb Q$ has a unique solution is violated by popular ambiguity sets such as the Wasserstein ambiguity set \cite{ref:DROtutorial-19}. It is satisfied, however, by the Kullback-Leibler \cite{hu2013kullback} and Sinkhorn \cite{wang2021sinkhorn} ambiguity sets.

\subsection{Gradient Projection Phenomenon}\label{sec:additional-considerations}
Despite their conceptual merits, optimization and projection layers in the decision map $m_w$ may impede the convergence of gradient-based training algorithms. Indeed, if the current iterate~$w\notin\mathcal A$ is {\em in}feasible, then optimization and projection layers in $m_w$ have a tendency to push the gradient of $\varphi(w)$ to a subspace of $\mathbb R^d$ that is (approximately) perpendicular to the shortest path from~$w$ to~$\mathcal A$. Thus, gradient-based methods like Algorithm~\ref{algo:training} may circle around $\mathcal A$ without ever reaching a feasible point. To our best knowledge, this phenomenon has not yet been studied or even recognized. 


As an example, consider an instance of problem~\eqref{eq:SP-static} with $\mathcal{X}=\mathbb R$, $\mathcal{Y}=\mathbb{R}^2$, $\mathcal A=\{a\in\mathbb R^2:\|a\|_2\leq 1\}$ and $\ell(y,a)=\|y-a\|_2^2$. Assume further that $\PP{X}=\delta_1$ and $\PP{Y|X}=\delta_{(0,\frac{1}{2})}$. In this case the constant decision map $m(X)=(0,\frac{1}{2})$ is optimal in~\eqref{eq:SP-static}. We now approximate~\eqref{eq:SP-static} by problem~\eqref{eq:bayesian-interpretation}, which minimizes over parametric decision maps of the form $m_w=p\circ f_w$, and we assume that the prescriptor consists of an optimization layer that maps any feature~$r\in\mathbb R^2$ to
$$
p(r)=\arg\min_{a\in\mathcal A}\|a-r\|_2,
$$
while the feature extractor parametrized by $w\in\mathbb{R}^2$ maps any observation~$x$ to a feature~$f_w(x)=xw=w$ $\mathbb P$-almost surely.
%
%
Training the decision map $m_w$ via Algorithm~\ref{algo:training} with an initial iterate~$w_0\notin \mathcal A$ generates a sequence of iterates that stay outside of~$\mathcal A$ as visualized in Figure~\ref{fig:cvxlayer-problem} (solid blue line).
The corresponding predictions satisfy $m_{w_k}(X)=p(w_k)$ $\mathbb P$-almost surely. Thus, they stay on the boundary of $\mathcal A$ (dashed blue line). If $w_0\in\mathcal A$, on the other hand, then the iterates converge to the global minimum (solid red line). In this case, the corresponding predictions satisfy $m_{w_k}(X)=p(w_k)=w_k$ $\mathbb P$-almost surely. Thus, they coincide with the underlying iterates (dashed red line).

\begin{figure}
    \centering
    \begin{picture}(120,120)
\put(0,0){\includegraphics[trim={0.2cm 0cm 0.5cm 0.5cm}, clip, width=0.5\linewidth]{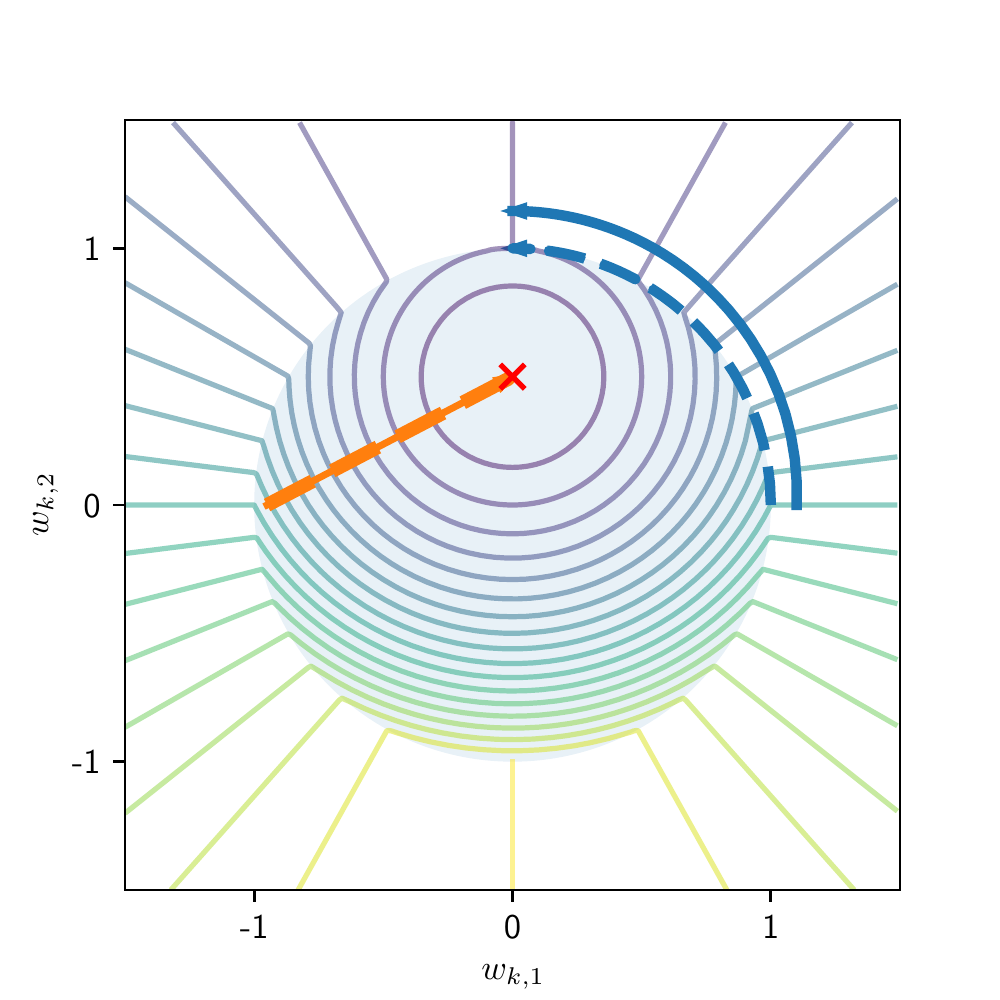}}
\put(15,60){\color{C1}$w_0$}
\put(63,70){\color{C1}$w_K$}
\put(100,55){\color{C0}$w_0$}
\put(63,102){\color{C0}$w_K$}
\put(55,42){$\mathcal{A}$}
\end{picture}
    \caption{Contours of the expected loss $\varphi(w)=\mathbb E[\ell(Y,m_w(X))]$. If the gradient $\nabla\varphi(w)$ exists, then it is perpendicular to the contours of $\varphi(w)$. Thus, the iterates of Algorithm~\ref{algo:training} converge to the global minimum $(0,\frac{1}{2})$ of $\varphi(w)$ if $w_0\in\mathcal A$ (solid red line) or to a local minimum in $\{(0,t):t>1\}$ if $w_0\notin\mathcal A$ (solid blue line).}
    \label{fig:cvxlayer-problem}
    \vspace{-0.2cm}
\end{figure}
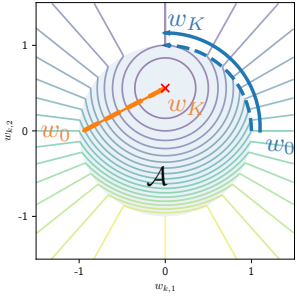
\section{Experiments}
We benchmark the discussed algorithms for end-to-end learning against simple baselines as well as the predict-then-optimize approach \cite{mivsic2020data} in the context of a newsvendor problem and an economic dispatch problem.
Implementation details are given in Appendix~\ref{app:experiment-details}, and the code underlying all experiments is provided on GitHub.\footnote{\url{https://github.com/RAO-EPFL/end2end-SO}}

\subsection{Newsvendor Problem}\label{sec:exp-newsvendor}
We first compare the Bayesian model addressed by Algorithm~\ref{algo:training} against the ERM and DRO models. To this end, we consider the decision problem of the seller of a perishable good ( e.g., a newspaper). At the beginning of each day, the newsvendor buys a number $a\in \mathcal{A}=\{1,\dots,d\}$ of items from the supplier at a wholesale price $p>0$. During the day she sells the items at the retail price $q>p$ until the supply~$a$ is exhausted or the random demand $Y\in\mathcal{Y}=\mathcal{A}$ is covered. The salvage value of unsold items is~$0$. Hence, the newsvendor's total cost amounts to $\ell(Y,a) = pa - q\min\{a,Y\}$. We assume that the unobserved confounder~$Z\in\mathbb R^d$ represents the demand distribution, that is, $Z_j = \mathbb{P}_{Y|Z}(Y=j)$ for all $j=1,\dots, d$. If the newsvendor observes $N$ independent historical demands $X=[\widehat Y_1,\ldots, \widehat Y_N]$ sampled from~$\mathbb P_{Y|Z}$, then she aims to solve the following instance of~\eqref{eq:SP}
\begin{equation*}
    \min_{a\in \mathcal{A}}\mathbb{E}_{}[\ell(Y,a)|X] =  \min_{a\in \mathcal{A}} \textstyle \sum_{j=1}^d \ell(j,a) \mathbb{E}_{}[Z_j|X].
\end{equation*}
In the following we fix a common neural network architecture and train decision maps via Algorithms~\ref{algo:training}, \ref{algo:training-erm} and \ref{algo:training-dro} with $K=5\times 10^6$ samples. Specifically, in Algorithm~\ref{algo:training-dro} we set $\mathcal U(X)$ to the Kullback-Leibler ambiguity set of radius~$0.025$ around the empirical distribution on the demand samples contained in~$X$. We designate the strategies output by the three algorithms as ``NN\_BAY", ``NN\_ERM" and ``NN\_DRO", respectively. The strategies obtained by solving the Bayesian problem~\eqref{eq:SP}, the ERM problem~\eqref{eq:erm} and the DRO problem~\eqref{eq:dro} {\em exactly} are designated as ``BAY", ``ERM" and ``DRO", respectively. Finally, the strategy output by an oracle with perfect knowledge of~$Z$ is designated as ``True".

Figure~\ref{fig:cdf_comparison} visualizes the out-of-sample cumulative distribution functions of the profit (negative loss) generated by different data-driven strategies. In line with Theorems~\ref{thm:equivalence:algo:1}, \ref{thm:equivalence:algo:2} and~\ref{thm:equivalence:algo:3}, the neural network-based decisions display a similar performance as the corresponding optimization-based decisions. 
Figure~\ref{fig:cdf_comparison} further shows that a correct choice of the prior is crucial for the Bayesian strategies (BAY and NN\_BAY). Finally, while the expected profits generated by different strategies are similar, the lower tails of the profit distributions vary dramatically. 
For example, Figure~\ref{sfig:NW:correct:prior} indicates that the risk of a loss (negative profit) is almost $10$ times higher under the ERM strategy than under any other strategy. Appendix~\ref{app:experiment-details} provides a more detailed comparison between the neural network-based strategies and the corresponding optimization-based strategies.
\begin{figure}
    \centering
    \begin{subfigure}[b]{0.45\linewidth}
    \centering
    \includegraphics[width=\linewidth]{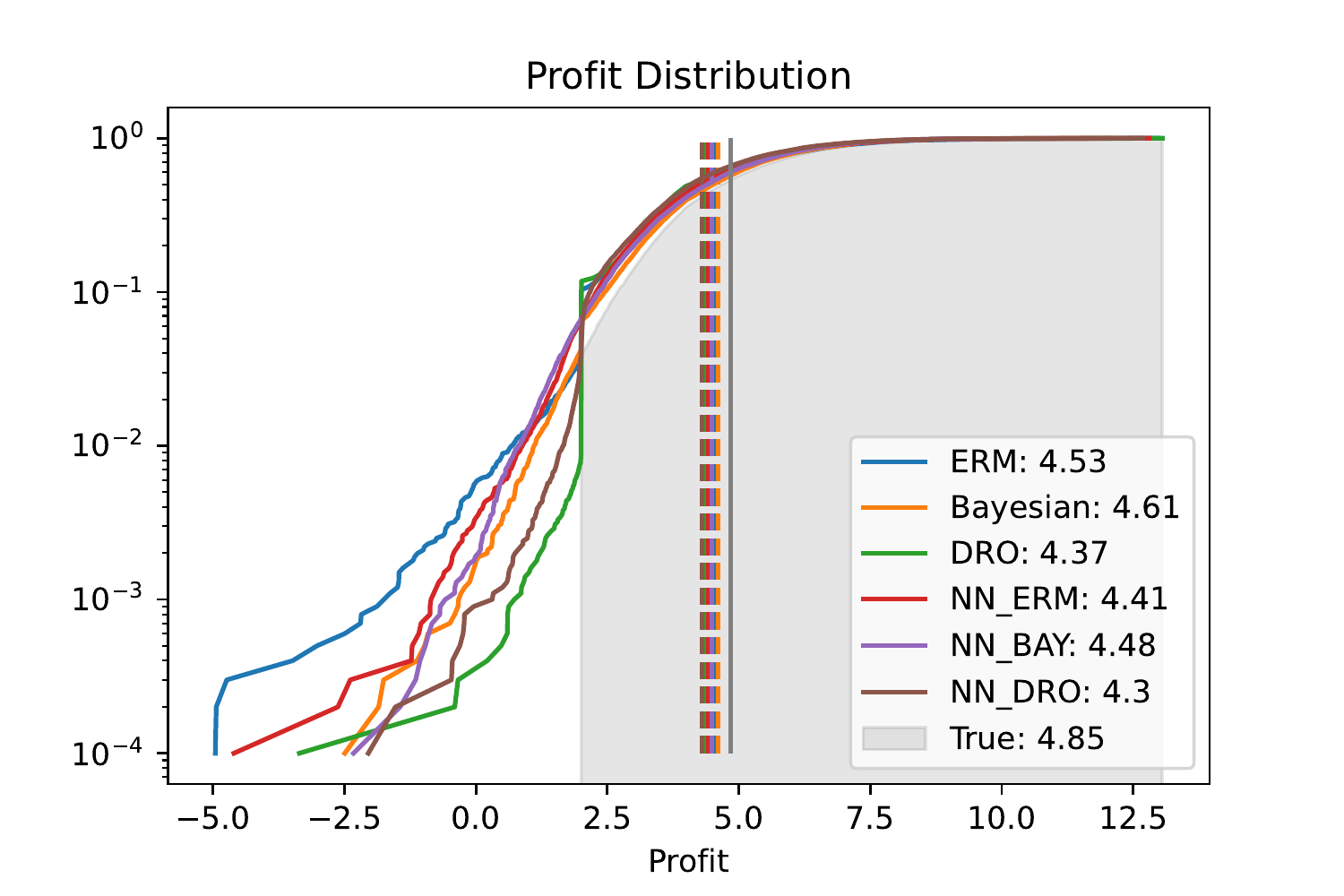}
    \caption{Correct Prior}
    \label{sfig:NW:correct:prior}
    \end{subfigure}
    \hfill
    \begin{subfigure}[b]{0.45\linewidth}
    \centering
    \includegraphics[width=\linewidth]{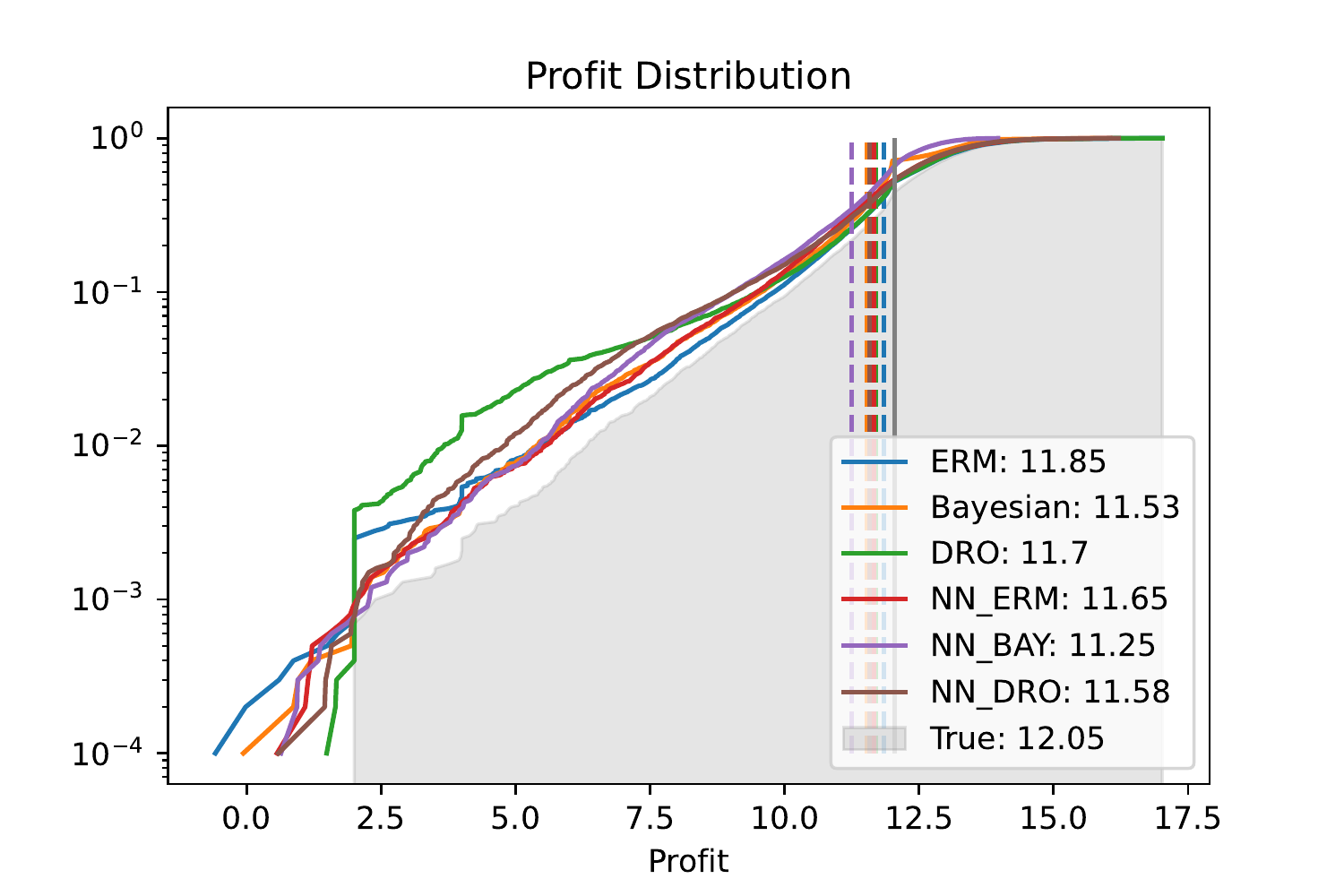}
    \caption{Incorrect Prior}
    \label{sfig:NW:incorrect:prior1}
    \end{subfigure}
    \caption{Cumulative distribution functions (solid lines) and expected values (dashed lines) of the out-of-sample profit generated by different data-driven strategies. 
    }
    \label{fig:cdf_comparison}
    \vspace{-0.2cm}
\end{figure}

\subsection{Economic Dispatch Problem}\label{sec:exp-economicdispatch}
In the second experiment, we revisit the stylized economic dispatch problem described in Section~\ref{sec:endtoend}. Here we assume that there are $J=6$ traditional generators and that the penalty for unmet demand amounts to $p=100$. The capacities $\Bar{a}_j$ of the six generators are given by $1$, $0.5$, $1$, $1$, $1$ and $0.5$, and the respective unit production costs $c_j$ are set to $15$, $20$, $15$, $20$, $30$ and $25$. We use historical wind power production and weather records\footnote{\url{https://www.kaggle.com/datasets/theforcecoder/wind-power-forecasting}} as samples from $\mathbb P_{(X,Y)}$.

\textbf{Observations} Contextual stochastic programs account for side information such as wind speed measurements or weather forecasts that can help to make better decisions. Any such side information is captured by the random variable~$X$, which is observable when the generation decisions $a\in\mathcal A$ must be selected. In the following experiment we distinguish three different possible observations.
The myopic observation \textbf{(Myopic)} consists of current temperature, wind speed and wind direction measurements. 
The incomplete myopic observation \textbf{(Myopic incomp.)} consists only of the current wind direction. The most informative observation \textbf{(Historical)} consists of the current temperature, wind speed and wind direction measurements as well as of the temperature, wind speed, wind direction, and wind energy production measurements at the last two timesteps. 

\textbf{Baselines} We compare the performance of all data-driven strategies to be described below against that of an ideal baseline strategy that has oracle access to the future wind power production level. This oracle strategy thus observes~$X=Y$, in which case the stochastic program~\eqref{eq:SP} collapses to a deterministic optimization problem. The resulting strategy is {\em infeasible} in practice and serves merely as a basis for comparison. A related {\em feasible} baseline strategy is obtained by pretending that $Y$ equals the wind energy production quantity observed in the last period and by solving the corresponding deterministic optimization problem akin to~\eqref{eq:SP}.

\textbf{Maximum Likelihood Estimation (MLE)} We construct a predict-then-optimize strategy by solving a least squares regression model for predicting the wind energy production level $Y$ and then solving the deterministic version of problem~\eqref{eq:SP} corresponding to this point prediction.

\textbf{End-to-End (E2E)} We also design end-to-end learning strategies, which are obtained by using Algorithm~\ref{algo:training} to train the feature extractor together with the prescriptor. We distinguish two different neural network architectures. On the one hand, we can define the prescriptor as an optimization problem layer (OPL) and the feature extractor as a predictor of $Y$. A Softplus activation function in the output layer of the feature extractor ensures that the prediction of~$Y$ is nonnegative. On the other hand, we can define the prescriptor as a constraint-aware layer (CAL), which uses rescaled sigmoid activation functions to force its output into~$\mathcal{A}$.

\begin{table}[t]
\centering
\caption{Average test costs generated by different data-driven strategies for the economic dispatch problem.}
\label{tab:edp-results}
\vspace{0.1cm}
\scalebox{0.75}{
\begin{tabular}{ll|ll}
Approach                          & Observation & 10 minute freq. & 30 minute freq. \\\hline\hline
\multirow{2}{*}{Baseline}         & Oracle            & 60.69              & 60.69             \\
                                  & Lag-1         & 64.96              & 67.71              \\
\hline\multirow{3}{*}{MLE}              & Myopic            & 281.44(1.65)      & 280.10(0.75)       \\
                                  & Myopic Incomp.    & 348.41(0.0)        & 348.43(0.0)         \\
                                  & Historical        & 304.56(11.47)     & 277.74(2.38)      \\
\hline\multirow{3}{*}{E2E-CAL}          & Myopic            & 68.25(3.73)        & 66.57(3.07)       \\
                                  & Myopic Incomp.    & 72.62(0.00)       & 74.11(3.00)       \\
                                  & Historical        & 67.09(4.50)       & 77.06(5.59)        \\
\hline\multirow{3}{*}{E2E-OPL-Softplus} & Myopic            & 71.33(1.87)       & 72.53(0.13)       \\
                                  & Myopic Incomp.    & 72.60(0.01)       & 72.60(0.00)       \\
                                  & Historical        & 72.60(0.0)         & 72.60(0.0)         
\end{tabular}}
\vspace{-0.1cm}
\end{table}
Table~\ref{tab:edp-results} reports the test costs of the different strategies. All strategies are compared under different data frequency models. That is, data is either observed every 10 minutes or every 30 minutes.
We observe that the MLE method incurs the highest costs. This may be attributed to the symmetric training loss, which ignores that it is better to underestimate energy production because unmet demand is heavily penalized. If new data is observed every 10 minutes, then the lag-$1$ baseline performs best, while the end-to-end strategy with access to historical data and with constraint-aware layers performs best among all neural network-based approaches.  If new data is observed every 30 minutes, then the stationarity assumption is violated and the lag-$1$ baseline is no longer competitive. In this case, the most useful observations are the wind speed measurements, which cannot be leveraged by the baselines. End-to-end learning strategies based on optimization layers often fail to train, which we attribute to the gradient projection problem. The myopic observation is most useful, but the incomplete myopic observation does not contain enough information to accurately estimate the wind power production. The historical information is overly rich and therefore leads to overfitting~\cite{ying2019overview}.

\bibliography{references}

\begin{thebibliography}{64}
\providecommand{\natexlab}[1]{#1}
\providecommand{\url}[1]{\texttt{#1}}
\expandafter\ifx\csname urlstyle\endcsname\relax
  \providecommand{\doi}[1]{doi: #1}\else
  \providecommand{\doi}{doi: \begingroup \urlstyle{rm}\Url}\fi

\bibitem[Agrawal et~al.(2019)Agrawal, Amos, Barratt, Boyd, Diamond, and
  Kolter]{cvxpylayers2019}
Agrawal, A., Amos, B., Barratt, S., Boyd, S., Diamond, S., and Kolter, Z.
\newblock Differentiable convex optimization layers.
\newblock In \emph{Advances in Neural Information Processing Systems}, 2019.

\bibitem[Amodei et~al.(2016)Amodei, Ananthanarayanan, Anubhai, Bai, Battenberg,
  Case, Casper, Catanzaro, Cheng, Chen, et~al.]{amodei2016deep}
Amodei, D., Ananthanarayanan, S., Anubhai, R., Bai, J., Battenberg, E., Case,
  C., Casper, J., Catanzaro, B., Cheng, Q., Chen, G., et~al.
\newblock Deep speech 2: End-to-end speech recognition in english and mandarin.
\newblock In \emph{International {C}onference on {M}achine {L}earning}, 2016.

\bibitem[Audibert et~al.(2009)Audibert, Munos, and
  Szepesv{\'a}ri]{audibert2009exploration}
Audibert, J.-Y., Munos, R., and Szepesv{\'a}ri, C.
\newblock Exploration--exploitation tradeoff using variance estimates in
  multi-armed bandits.
\newblock \emph{Theoretical Computer Science}, 410\penalty0 (19):\penalty0
  1876--1902, 2009.

\bibitem[Barocas et~al.(2019)Barocas, Hardt, and Narayanan]{fair-ML19}
Barocas, S., Hardt, M., and Narayanan, A.
\newblock \emph{Fairness and Machine Learning}.
\newblock 2019.
\newblock \url{http://www.fairmlbook.org}.

\bibitem[Barron(1991)]{ref:Barron-91}
Barron, A.~R.
\newblock Approximation and estimation bounds for artificial neural networks.
\newblock In \emph{Fourth Annual Workshop on Computational Learning Theory},
  1991.

\bibitem[Baum \& Wilczek(1987)Baum and Wilczek]{baum1987supervised}
Baum, E. and Wilczek, F.
\newblock Supervised learning of probability distributions by neural networks.
\newblock In \emph{Neural {I}nformation {P}rocessing {S}ystems}, 1987.

\bibitem[Ben-Tal et~al.(2013)Ben-Tal, Den~Hertog, De~Waegenaere, Melenberg, and
  Rennen]{ben2013robust}
Ben-Tal, A., Den~Hertog, D., De~Waegenaere, A., Melenberg, B., and Rennen, G.
\newblock Robust solutions of optimization problems affected by uncertain
  probabilities.
\newblock \emph{Management Science}, 59\penalty0 (2):\penalty0 341--357, 2013.

\bibitem[Berger(2013)]{berger2013bayesian_opt}
Berger, J.~O.
\newblock \emph{Statistical Decision Theory and Bayesian Analysis}.
\newblock Springer, 2013.

\bibitem[Bertsimas \& Kallus(2020)Bertsimas and
  Kallus]{bertsimas2020predictive}
Bertsimas, D. and Kallus, N.
\newblock From predictive to prescriptive analytics.
\newblock \emph{Management Science}, 66\penalty0 (3):\penalty0 1025--1044,
  2020.

\bibitem[Buolamwini \& Gebru(2018)Buolamwini and Gebru]{buolamwini2018gender}
Buolamwini, J. and Gebru, T.
\newblock Gender shades: Intersectional accuracy disparities in commercial
  gender classification.
\newblock In \emph{Conference on {F}airness, {A}ccountability and
  {T}ransparency}, 2018.

\bibitem[Butler \& Kwon(2021)Butler and Kwon]{ref:Butler-21}
Butler, A. and Kwon, R.~H.
\newblock Integrating prediction in mean-variance portfolio optimization.
\newblock \emph{arXiv preprint arXiv.2102.09287}, 2021.

\bibitem[Cho et~al.(2014)Cho, Van~Merri{\"e}nboer, Bahdanau, and
  Bengio]{cho2014properties}
Cho, K., Van~Merri{\"e}nboer, B., Bahdanau, D., and Bengio, Y.
\newblock On the properties of neural machine translation: Encoder-decoder
  approaches.
\newblock \emph{arXiv preprint arXiv:1409.1259}, 2014.

\bibitem[Chu et~al.(2011)Chu, Li, Reyzin, and Schapire]{chu2011contextual}
Chu, W., Li, L., Reyzin, L., and Schapire, R.
\newblock Contextual bandits with linear payoff functions.
\newblock In \emph{International Conference on Artificial Intelligence and
  Statistics}, 2011.

\bibitem[Cournot(1897)]{cournot1897researches}
Cournot, A.~A.
\newblock \emph{Researches into the Mathematical Principles of the Theory of
  Wealth}.
\newblock Macmillan Company, 1897.

\bibitem[Cybenko(1989)]{cybenko1989approximation}
Cybenko, G.
\newblock Approximation by superpositions of a sigmoidal function.
\newblock \emph{Mathematics of {C}ontrol, {S}ignals and {S}ystems}, 2\penalty0
  (4):\penalty0 303--314, 1989.

\bibitem[Delage \& Ye(2010)Delage and Ye]{delage2010distributionally}
Delage, E. and Ye, Y.
\newblock Distributionally robust optimization under moment uncertainty with
  application to data-driven problems.
\newblock \emph{Operations Research}, 58\penalty0 (3):\penalty0 595--612, 2010.

\bibitem[Delalleau \& Bengio(2011)Delalleau and Bengio]{ref:Delalleau-11}
Delalleau, O. and Bengio, Y.
\newblock Shallow vs. deep sum-product networks.
\newblock In \emph{Advances in Neural Information Processing Systems}, 2011.

\bibitem[Donti et~al.(2017)Donti, Amos, and Kolter]{donti2017end2end_constr}
Donti, P.~L., Amos, B., and Kolter, J.~Z.
\newblock Task-based end-to-end model learning in stochastic optimization.
\newblock In \emph{International Conference on Neural Information Processing
  Systems}, 2017.

\bibitem[Durrett(2010)]{durrett_book}
Durrett, R.
\newblock \emph{Probability: Theory and Examples}.
\newblock Cambridge University Press, 2010.

\bibitem[Elmachtoub \& Grigas(2022)Elmachtoub and Grigas]{ref:Elmachtoub-22}
Elmachtoub, A.~N. and Grigas, P.
\newblock Smart “predict, then optimize”.
\newblock \emph{Management Science}, 68\penalty0 (1):\penalty0 9–26, 2022.

\bibitem[Elman(1990)]{elman1990finding}
Elman, J.~L.
\newblock Finding structure in time.
\newblock \emph{Cognitive Science}, 14\penalty0 (2):\penalty0 179--211, 1990.

\bibitem[Fu et~al.(2018)Fu, Wang, Tsao, Lu, and Kawai]{fu2018end}
Fu, S.-W., Wang, T.-W., Tsao, Y., Lu, X., and Kawai, H.
\newblock End-to-end waveform utterance enhancement for direct evaluation
  metrics optimization by fully convolutional neural networks.
\newblock \emph{IEEE/ACM Transactions on Audio, Speech, and Language
  Processing}, 26\penalty0 (9):\penalty0 1570--1584, 2018.

\bibitem[Graves \& Jaitly(2014)Graves and Jaitly]{graves2014towards}
Graves, A. and Jaitly, N.
\newblock Towards end-to-end speech recognition with recurrent neural networks.
\newblock In \emph{International {C}onference on {M}achine {L}earning}, 2014.

\bibitem[He et~al.(2016)He, Zhang, Ren, and Sun]{he2016deep}
He, K., Zhang, X., Ren, S., and Sun, J.
\newblock Deep residual learning for image recognition.
\newblock In \emph{IEEE Conference on Computer Vision and Pattern Recognition},
  2016.

\bibitem[Hochreiter \& Schmidhuber(1997)Hochreiter and
  Schmidhuber]{hochreiter1997long}
Hochreiter, S. and Schmidhuber, J.
\newblock Long short-term memory.
\newblock \emph{Neural computation}, 9\penalty0 (8):\penalty0 1735--1780, 1997.

\bibitem[Hu \& Hong(2013)Hu and Hong]{hu2013kullback}
Hu, Z. and Hong, L.~J.
\newblock Kullback-{L}eibler divergence constrained distributionally robust
  optimization.
\newblock \emph{Optimization Online}, 2013.

\bibitem[Kalman(1960)]{kalmanfilter}
Kalman, R.~E.
\newblock A new approach to linear filtering and prediction problems.
\newblock \emph{Transactions of the ASME--Journal of Basic Engineering},
  82\penalty0 (Series D):\penalty0 35--45, 1960.

\bibitem[Kingma \& Ba(2015)Kingma and Ba]{kingma2015adam}
Kingma, D. and Ba, J.
\newblock {Adam}: A method for stochastic optimization.
\newblock In \emph{International Conference on Learning Representations}, 2015.

\bibitem[Kiran et~al.(2021)Kiran, Sobh, Talpaert, Mannion, Al~Sallab, Yogamani,
  and P{\'e}rez]{kiran2021deep}
Kiran, B.~R., Sobh, I., Talpaert, V., Mannion, P., Al~Sallab, A.~A., Yogamani,
  S., and P{\'e}rez, P.
\newblock Deep reinforcement learning for autonomous driving: A survey.
\newblock \emph{IEEE Transactions on Intelligent Transportation Systems},
  23\penalty0 (6):\penalty0 4909--4926, 2021.

\bibitem[Kline \& Berardi(2005)Kline and Berardi]{kline2005revisiting}
Kline, D.~M. and Berardi, V.~L.
\newblock Revisiting squared-error and cross-entropy functions for training
  neural network classifiers.
\newblock \emph{Neural Computing \& Applications}, 14\penalty0 (4):\penalty0
  310--318, 2005.

\bibitem[Kuhn et~al.(2019)Kuhn, {Mohajerin Esfahani}, Nguyen, and
  Shafieezadeh-Abadeh]{ref:DROtutorial-19}
Kuhn, D., {Mohajerin Esfahani}, P., Nguyen, V.~A., and Shafieezadeh-Abadeh, S.
\newblock Wasserstein distributionally robust optimization: Theory and
  applications in machine learning.
\newblock In \emph{Operations Research \& Management Science in the Age of
  Analytics}, pp.\  130--166. INFORMS, 2019.

\bibitem[Lattimore \& Szepesv\`ari(2020)Lattimore and
  Szepesv\`ari]{ref:book:Csaba}
Lattimore, T. and Szepesv\`ari, C.
\newblock \emph{Bandit Algorithms}.
\newblock Cambridge University Press, 2020.

\bibitem[LeCun et~al.(1989)LeCun, Boser, Denker, Henderson, Howard, Hubbard,
  and Jackel]{lecun1989lenet}
LeCun, Y., Boser, B., Denker, J.~S., Henderson, D., Howard, R.~E., Hubbard, W.,
  and Jackel, L.~D.
\newblock Backpropagation applied to handwritten zip code recognition.
\newblock \emph{Neural computation}, 1\penalty0 (4):\penalty0 541--551, 1989.

\bibitem[Li(2002)]{li2002macroeconomic}
Li, L.
\newblock Macroeconomic factors and the correlation of stock and bond returns.
\newblock \emph{Available at SSRN 363641}, 2002.

\bibitem[Long et~al.(2010)Long, Chapelle, Zhang, Chang, Zheng, and
  Tseng]{long2010active}
Long, B., Chapelle, O., Zhang, Y., Chang, Y., Zheng, Z., and Tseng, B.
\newblock Active learning for ranking through expected loss optimization.
\newblock In \emph{ACM SIGIR Conference on Research and Development in
  Information Retrieval}, 2010.

\bibitem[Lu et~al.(2017)Lu, Pu, Wang, Hu, and Wang]{NIPS2017_32cbf687}
Lu, Z., Pu, H., Wang, F., Hu, Z., and Wang, L.
\newblock The expressive power of neural networks: A view from the width.
\newblock In \emph{Advances in Neural Information Processing Systems}, 2017.

\bibitem[Markowitz \& Todd(2000)Markowitz and Todd]{markowitz2000mean}
Markowitz, H.~M. and Todd, G.~P.
\newblock \emph{Mean-variance analysis in portfolio choice and capital
  markets}.
\newblock John Wiley \& Sons, 2000.

\bibitem[Mi{\v{s}}i{\'c} \& Perakis(2020)Mi{\v{s}}i{\'c} and
  Perakis]{mivsic2020data}
Mi{\v{s}}i{\'c}, V.~V. and Perakis, G.
\newblock Data analytics in operations management: A review.
\newblock \emph{Manufacturing \& Service Operations Management}, 22\penalty0
  (1):\penalty0 158--169, 2020.

\bibitem[Murphy(2007)]{murphy2007conjugate}
Murphy, K.~P.
\newblock Conjugate {B}ayesian analysis of the {G}aussian distribution.
\newblock Technical report, University of British Columbia, 2007.

\bibitem[Papoulis \& Pillai(2002)Papoulis and Pillai]{papoulis2002probability}
Papoulis, A. and Pillai, S.~U.
\newblock \emph{Probability, Random Variables, and Stochastic Processes}.
\newblock McGraw-Hill, 2002.

\bibitem[Pezeshk(2003)]{pezeshk2003bayesian}
Pezeshk, H.
\newblock Bayesian techniques for sample size determination in clinical trials:
  a short review.
\newblock \emph{Statistical Methods in Medical Research}, 12\penalty0
  (6):\penalty0 489--504, 2003.

\bibitem[Rahimian \& Mehrotra(2019)Rahimian and Mehrotra]{ref:Rahimian-19}
Rahimian, H. and Mehrotra, S.
\newblock {Distributionally robust optimization: A review}.
\newblock \emph{arXiv preprint arXiv:1908.05659}, 2019.

\bibitem[Rockafellar \& Wets(2009)Rockafellar and
  Wets]{rockafellar2009variational}
Rockafellar, R.~T. and Wets, R. J.-B.
\newblock \emph{Variational Analysis}.
\newblock Springer, 2009.

\bibitem[Rumelhart et~al.(1986)Rumelhart, Hinton, and
  Williams]{rumelhart1986learning}
Rumelhart, D.~E., Hinton, G.~E., and Williams, R.~J.
\newblock Learning representations by back-propagating errors.
\newblock \emph{Nature}, 323\penalty0 (6088):\penalty0 533--536, 1986.

\bibitem[Shapiro et~al.(2021)Shapiro, Dentcheva, and
  Ruszczynski]{shapiro2021lectures}
Shapiro, A., Dentcheva, D., and Ruszczynski, A.
\newblock \emph{Lectures on Stochastic Programming: Modeling and Theory}.
\newblock SIAM, 2021.

\bibitem[Silver et~al.(2017)Silver, Schrittwieser, Simonyan, Antonoglou, Huang,
  Guez, Hubert, Baker, Lai, Bolton, et~al.]{silver2017mastering}
Silver, D., Schrittwieser, J., Simonyan, K., Antonoglou, I., Huang, A., Guez,
  A., Hubert, T., Baker, L., Lai, M., Bolton, A., et~al.
\newblock Mastering the game of go without human knowledge.
\newblock \emph{Nature}, 550\penalty0 (7676):\penalty0 354--359, 2017.

\bibitem[Silver et~al.(2018)Silver, Hubert, Schrittwieser, Antonoglou, Lai,
  Guez, Lanctot, Sifre, Kumaran, Graepel, et~al.]{silver2018general}
Silver, D., Hubert, T., Schrittwieser, J., Antonoglou, I., Lai, M., Guez, A.,
  Lanctot, M., Sifre, L., Kumaran, D., Graepel, T., et~al.
\newblock A general reinforcement learning algorithm that masters chess, shogi,
  and go through self-play.
\newblock \emph{Science}, 362\penalty0 (6419):\penalty0 1140--1144, 2018.

\bibitem[Smith \& Winkler(2006)Smith and Winkler]{smith2006optimizer}
Smith, J.~E. and Winkler, R.~L.
\newblock The optimizer’s curse: Skepticism and postdecision surprise in
  decision analysis.
\newblock \emph{Management Science}, 52\penalty0 (3):\penalty0 311--322, 2006.

\bibitem[Stengel(1994)]{stengel1994optimal}
Stengel, R.~F.
\newblock \emph{Optimal Control and Estimation}.
\newblock Courier Corporation, 1994.

\bibitem[Uysal et~al.(2021{\natexlab{a}})Uysal, Li, and Mulvey]{ref:Uysal-21}
Uysal, A.~S., Li, X., and Mulvey, J.~M.
\newblock End-to-end risk budgeting portfolio optimization with neural
  networks.
\newblock \emph{arXiv preprint arXiv.2107.04636}, 2021{\natexlab{a}}.

\bibitem[Uysal et~al.(2021{\natexlab{b}})Uysal, Li, and Mulvey]{uysal2021end}
Uysal, A.~S., Li, X., and Mulvey, J.~M.
\newblock End-to-end risk budgeting portfolio optimization with neural
  networks.
\newblock \emph{arXiv preprint arXiv:2107.04636}, 2021{\natexlab{b}}.

\bibitem[Vaswani et~al.(2017)Vaswani, Shazeer, Parmar, Uszkoreit, Jones, Gomez,
  Kaiser, and Polosukhin]{vaswani2017attention}
Vaswani, A., Shazeer, N., Parmar, N., Uszkoreit, J., Jones, L., Gomez, A.~N.,
  Kaiser, {\L}., and Polosukhin, I.
\newblock Attention is all you need.
\newblock \emph{Advances in Neural Information Processing Systems}, 2017.

\bibitem[Vinyals et~al.(2019)Vinyals, Babuschkin, Czarnecki, Mathieu, Dudzik,
  Chung, Choi, Powell, Ewalds, Georgiev, et~al.]{vinyals2019grandmaster}
Vinyals, O., Babuschkin, I., Czarnecki, W.~M., Mathieu, M., Dudzik, A., Chung,
  J., Choi, D.~H., Powell, R., Ewalds, T., Georgiev, P., et~al.
\newblock Grandmaster level in {StarCraft II} using multi-agent reinforcement
  learning.
\newblock \emph{Nature}, 575\penalty0 (7782):\penalty0 350--354, 2019.

\bibitem[Wan(1990)]{wan1990neural}
Wan, E.~A.
\newblock Neural network classification: A {B}ayesian interpretation.
\newblock \emph{IEEE Transactions on Neural Networks}, 1\penalty0 (4):\penalty0
  303--305, 1990.

\bibitem[Wang et~al.(2021{\natexlab{a}})Wang, Gao, and Xie]{wang2021sinkhorn}
Wang, J., Gao, R., and Xie, Y.
\newblock Sinkhorn distributionally robust optimization.
\newblock \emph{arXiv preprint arXiv:2109.11926}, 2021{\natexlab{a}}.

\bibitem[Wang et~al.(2011)Wang, Babenko, and Belongie]{wang2011end}
Wang, K., Babenko, B., and Belongie, S.
\newblock End-to-end scene text recognition.
\newblock In \emph{International Conference on Computer Vision}, 2011.

\bibitem[Wang et~al.(2021{\natexlab{b}})Wang, Magn{\'u}sson, and
  Johansson]{wang2021convergence}
Wang, X., Magn{\'u}sson, S., and Johansson, M.
\newblock On the convergence of step decay step-size for stochastic
  optimization.
\newblock \emph{Advances in Neural Information Processing Systems},
  2021{\natexlab{b}}.

\bibitem[Wang et~al.(2017)Wang, Skerry-Ryan, Stanton, Wu, Weiss, Jaitly, Yang,
  Xiao, Chen, Bengio, Le, Agiomyrgiannakis, Clark, and
  Saurous]{wang2017tacotron}
Wang, Y., Skerry-Ryan, R., Stanton, D., Wu, Y., Weiss, R.~J., Jaitly, N., Yang,
  Z., Xiao, Y., Chen, Z., Bengio, S., Le, Q., Agiomyrgiannakis, Y., Clark, R.,
  and Saurous, R.~A.
\newblock Tacotron: A fully end-to-end text-to-speech synthesis model.
\newblock \emph{arXiv preprint arXiv:1703.10135}, 2017.

\bibitem[Wiesemann et~al.(2014)Wiesemann, Kuhn, and
  Sim]{wiesemann2014distributionally}
Wiesemann, W., Kuhn, D., and Sim, M.
\newblock Distributionally robust convex optimization.
\newblock \emph{Operations Research}, 62\penalty0 (6):\penalty0 1358--1376,
  2014.

\bibitem[Wong \& Fuller(2007)Wong and Fuller]{wong2007pricing}
Wong, S. and Fuller, J.~D.
\newblock Pricing energy and reserves using stochastic optimization in an
  alternative electricity market.
\newblock \emph{IEEE Transactions on Power Systems}, 22\penalty0 (2):\penalty0
  631--638, 2007.

\bibitem[Ying(2019)]{ying2019overview}
Ying, X.
\newblock An overview of overfitting and its solutions.
\newblock In \emph{Journal of {P}hysics: {C}onference {S}eries}, 2019.

\bibitem[Zhang et~al.(2021)Zhang, Zhang, Cucuringu, and
  Zohren]{zhang2021universal}
Zhang, C., Zhang, Z., Cucuringu, M., and Zohren, S.
\newblock A universal end-to-end approach to portfolio optimization via deep
  learning.
\newblock \emph{arXiv preprint arXiv:2111.09170}, 2021.

\bibitem[Zhang et~al.(1988)Zhang, Tanida, Itoh, and Ichioka]{zhang1988shift}
Zhang, W., Tanida, J., Itoh, K., and Ichioka, Y.
\newblock Shift-invariant pattern recognition neural network and its optical
  architecture.
\newblock In \emph{Annual Conference of the Japan Society of Applied Physics},
  1988.

\bibitem[Zhang et~al.(2020)Zhang, Zohren, and Roberts]{ref:Zhang-20}
Zhang, Z., Zohren, S., and Roberts, S.
\newblock Deep learning for portfolio optimization.
\newblock \emph{The Journal of Financial Data Science}, 2\penalty0
  (4):\penalty0 8--20, 2020.

\end{thebibliography}
\bibliographystyle{icml2023}

\newpage
\appendix
\onecolumn
\section{Proofs of Section~\ref{sec:model-architecture}}
\begin{proof}[Proof of Proposition~\ref{prop:universal_approximation}]
We first prove Assertion~(i). As $p$ is Lipschitz continuous, we have
$$
\|p(r_1)-p(r_2)\|\leq L_p\|r_1-r_2\|\quad \forall r_1, r_2\in \mathcal{R}.
$$
This is notably also true for $r_1=f(x)$ and $r_2=f_w(x)$, in which case we obtain
$$
\|m(x)-m_w(x)\| = \|p(f(x))-p(f_w(x))\|\leq L_p\|f(x)-f_w(x)\|\leq L\varepsilon\quad\forall x\in\mathcal{X},
$$
where the last inequality follows from the assumption about $f_w(x)$. The claim now follows by maximizing the left hand side across all $x\in\mathcal X$. As for Assertion~(ii), the first part of the proof readily implies that
\begin{align*}
\mathbb E\left[ \|m(x)-m_w(x)\|^q \right] \leq &L_p^q \mathbb E\left[\|f(x)-f_w(x)\|^p\right] \leq L_p^q\varepsilon,
\end{align*}
where the last inequality follows from the the assumption about $f_w(x)$.
\end{proof}
\begin{proof}[Proof of Corollary~\ref{cor:universal-approx-loss}]
    Set $\delta=\varepsilon/(L_p L_\ell)$. By~\cite{cybenko1989approximation}, there exists a neural network $f_w$ with a single hidden layer of sufficient width and with sigmoid activation functions such that $\sup_{x\in\mathcal{X}}\|f(x)-f_w(x)\|\leq \delta$. By the Lipschitz-continuity of $p$ and $\ell$, we thus obtain
    \begin{align*}
        \sup_{x\in\mathcal{X},y\in\mathcal{Y}} |\ell(y,m(x))-\ell(y,m_w(x))|\leq 
        \sup_{x\in\mathcal{X}} L_\ell\|m(x)-m_w(x)\|\leq \sup_{x\in\mathcal{X}}L_pL_\ell\|f(x)-f_w(x)\|\leq\varepsilon.
    \end{align*}
    The claim now follows because the expected value of a non-negative random variable is upper bounded by its supremum.
\end{proof}

\section{Proofs of Section~\ref{sec:training-process}}




\begin{proof}[Proof of Theorem~\ref{thm:equivalence:algo:1}]
As for Assertion~(i), note that
the expected value of $g_k$ conditional on the last iterate $w_{k-1}$ satisfies
\begin{align*}
    \EE{}[g_k|w_{k-1}] &= \mathbb E[\nabla_w \ell(Y_k,m_w(X_{k}))|_{w=w_{k-1}}|w_{k-1}] \\
    &= \nabla_w \mathbb E[ \ell(Y_k,m_w(X_{k}))|w_{k-1}] |_{w=w_{k-1}}\\
    & = \nabla_w\EE{}\big[\ell(Y,m_w(X))]|_{w=w_{k-1}}=\nabla\varphi(w_{k-1}),
\end{align*}
where the first equality follows from the definition of $g_k$. The second equality holds because the gradient with respect to $w$ and the expectation conditional on $w_{k-1}$ can be interchanged thanks to the dominated convergence theorem, which applies thanks to Assumption~\ref{ass:smoothness}. The third equality, finally, exploits the independence of $(X_k,Y_k)$ and $w_{k-1}$, and the last equality follows from the definition of~$\varphi$.
Assertion~(ii) follows from~\citep[Theorem~3.5]{wang2021convergence}, which applies again thanks to Assumption~\ref{ass:smoothness}. Indeed, $g_k$ constitutes an unbiased gradient estimator for $\varphi(w)$ at $w=w_{k-1}$ thanks to Assertion~(i). Assumption~\ref{ass:smoothness} further implies that $\varphi(w)$ is $L$-smooth for some $L<\infty$. In addition, the stochastic gradients $g_k$ have bounded variance because they are themselves bounded by Assumption~\ref{ass:smoothness}. Finally, for $\varphi^\star=\min_{w\in\mathbb R^d}\varphi(w)$, we have that $\EE{}[\varphi(w_k)-\varphi^\star]$ is bounded because $\ell$ is bounded. Thus, the claim follows.
\end{proof}
The next lemma establishes the equivalence of problems~\eqref{eq:erm} and~\eqref{eq:ERM-static}.

\begin{lemma}[Equivalence of~\eqref{eq:erm} and~\eqref{eq:ERM-static}]\label{lemma:erm-equivalence}
Problems~\eqref{eq:erm} and~\eqref{eq:ERM-static} are equivalent in the sense that if $m^\star$ solves~\eqref{eq:ERM-static}, then $m^\star(X)$ solves~\eqref{eq:erm} $\mathbb P$-almost surely and vice versa.\end{lemma}
\begin{proof}[Proof of Lemma~\ref{lemma:erm-equivalence}]
By the law of iterated conditional expectations \citep[Theorem 5.1.6.]{durrett_book}, \eqref{eq:ERM-static} can be recast as
\begin{equation}
    \label{eq:interchangeability1}
    \min_{m\in\mathcal{M}}\EE{}\left[\EE{}\left[ \left. \frac{1}{N}\sum_{n=1}^N\ell(\widehat{Y}_n,  m(X))\,\right|X\right]\right].
\end{equation}
Next, by the interchangeability principle~\citep[Theorem~14.60]{rockafellar2009variational}, minimizing over all measurable decision maps~$m\in \mathcal M$ outside of the outer expectation is equivalent to minimizing over all decisions $a\in\mathcal A$ inside the outer expectation. Hence, the above optimization problem is equivalent to
\begin{equation}
    \label{eq:interchangeability2}
    \EE{}\left[ \min_{a\in\mathcal{A}} \EE{}\left[ \left. \frac{1}{N}\sum_{n=1}^N\ell(\widehat{Y}_n,  a)\,\right|X\right]\right].
\end{equation}
In addition, the interchangeability principle also implies that $m^\star$ solves the minimization problem over $m\in\mathcal M$ in~\eqref{eq:interchangeability1} if and only if $a^\star=m^\star(X)$ solves the minimization problem over $a\in\mathcal A$ in~\eqref{eq:interchangeability2} $\mathbb P$-almost surely. As $X = [\widehat Y_1,\ldots, \widehat Y_N]$, the samples $\widehat Y_n$ are measurable with respect to the $\sigma$-algebra generated by $X$ for all $i=1,\ldots, N$. Therefore, the inner minimization problem in~\eqref{eq:interchangeability2} is in fact equivalent to the ERM problem~\eqref{eq:erm}.
Thus, the claim follows.
\end{proof}

\begin{proof}[Proof of Theorem~\ref{thm:equivalence:algo:2}]
The proof widely parallels that of Theorem~\ref{thm:equivalence:algo:1}. Details are omitted for brevity.
\end{proof}

Similar to Lemma~\ref{lemma:erm-equivalence}, we provide a Lemma to outline the equivalence of problems~\eqref{eq:dro} and~\eqref{eq:DRO-static}.

\begin{lemma}[Equivalence of~\eqref{eq:dro} and~\eqref{eq:DRO-static}]\label{lemma:dro-equivalence}
Problems~\eqref{eq:dro} and~\eqref{eq:DRO-static} are equivalent in the sense that if $m^\star$ solves~\eqref{eq:DRO-static}, then $m^\star(X)$ solves~\eqref{eq:dro} $\mathbb P$-almost surely and vice versa.\end{lemma}
\begin{proof}[Proof of Lemma~\ref{lemma:dro-equivalence}]
By the law of iterated conditional expectations \citep[Theorem 5.1.6.]{durrett_book}, \eqref{eq:DRO-static} can be recast as
\begin{equation}
    \label{eq:interchangeability3}
    \min_{m\in\mathcal{M}}\EE{}\left[\EE{}\left[ \left.\max_{\mathbb Q\in\mathcal{U}(X)}  \int_{\mathcal Y}\ell(y,m(X)) \, {\rm d}\mathbb Q(y)\,\right|X \right]\right],
\end{equation}
and by the interchangeability principle~\citep[Theorem~14.60]{rockafellar2009variational}, this is equivalent to
\begin{equation}
    \label{eq:interchangeability4}
    \EE{}\left[\min_{a\in\mathcal{A}}\EE{}\left[ \left.\max_{\mathbb Q\in\mathcal{U}(X)}  \int_{\mathcal Y}\ell(y,a) \, {\rm d}\mathbb Q(y)\,\right|X \right]\right] =\EE{}\left[\min_{a\in\mathcal{A}} \max_{\mathbb Q\in\mathcal{U}(X)}  \int_{\mathcal Y}\ell(y,a) \, {\rm d}\mathbb Q(y)\right].
\end{equation}
The last equality holds because $X = [\widehat Y_1,\ldots, \widehat Y_N]$. The interchangeability principle also implies that $m^\star$ solves the minimization problem over $m\in\mathcal M$ in~\eqref{eq:interchangeability3} if and only if $a^\star=m^\star(X)$ solves the minimization problem over $a\in\mathcal A$ in~\eqref{eq:interchangeability4} $\mathbb P$-almost surely. This observation completes the proof.
\end{proof}

\begin{proof}[Proof of Theorem~\ref{thm:equivalence:algo:3}]
By Assumption~\ref{ass:smoothness}, the gradient $\nabla_w \int_{\mathcal Y}\ell(y,m_w(X_{k})){\rm d}\mathbb{Q}(y)$ exists and is continuous in $w$ for every $\mathbb{Q}\in \mathcal U(X)$ and for every $k=1,\ldots, K$. Since the maximization problem over~$\mathbb Q$ in~\eqref{eq:bayesian-interpretation-dro} has $\mathbb P$-almost surely a unique solution, Danskin's Theorem~\citep[Theorem~7.21]{shapiro2021lectures} implies that 
$$
    g_k= \nabla_w \left[\left.\int_{\mathcal Y}\ell(y, m_{w}(X_{k}) \,{\rm d}\mathbb{Q}_k^\star(y) \right]\right|_{w=w_{k-1}}
    = \nabla_w \left[\left.\max_{\mathbb{Q}\in\mathcal{U}(X_k)}\int_{\mathcal Y}\ell(y, m_{w}(X_{k}) \,{\rm d}\mathbb{Q}(y)\right]\right|_{w=w_{k-1}}.
$$
As $g_k$ is bounded, it has a bounded variance, and thus we can conclude that
$$
\EE{}[g_k|w_{k-1}]=\nabla_w \EE{}\left.\left[\max_{\mathbb{Q}\in\mathcal{U}(X_k)}\int\ell(y, m_{w}(X_{k})\, {\rm d}\mathbb{Q}(y)\right] \right|_{w=w_{k-1}}.
$$
This observation completes the proof.
\end{proof}

\section{Details on Experiments}\label{app:experiment-details}
\subsection{Details on the Minimum Mean-Square Estimation Experiment in Section~\ref{ssec:mean:estimation}}

The observation is given by $X=[\widehat Y_1,\ldots,\widehat Y_{20}]$, whose components $\widehat Y_n$ are sampled independently from $\PP{Y|Z}$. Thus, $Y$ and~$X$ are independent conditionally on~$Z$ as in Figure~\ref{fig:bayesian:standard}. Note, however, that both $Y$ and $X$ depend on $Z$, and thus they are not (unconditionally) independent. We solve problem~\eqref{eq:mean-estimation} with a Batch-SGD variant of Algorithm~\ref{algo:training}, where the training samples $\{(X_{k},Y_{k})\}_{k=1}^K$ are generated by first sampling $Z$ from the prior $\mathcal{N}(2,0.25)$. We then sample $X_k$ from $\PP{X|Z}$ and $Y_k$ from $\PP{Y|Z}$. 
The Batch-SGD algorithm runs over $50{,}000$ iterations with $100$ samples per batch to reduce the variance of the gradient updates. Therefore, a total of $K=5\times 10^6$ training samples are used for training. The neural network-based predictions $\widehat\mu_{\rm NN}$ are compared against the sample mean $\widehat\mu_{\rm ERM}$ and posterior mean $\widehat\mu_{\rm MMSE}$.
The posterior mean can be computed in closed form.
Indeed, since we use a conjugate prior for the mean of a Gaussian distribution, the solution of problem~\eqref{eq:mean-estimation} can be computed analytically. 
This is possible because $\PP{Y|X}=\mathcal{N}(\mu_\mathsf{Bayes},\sigma_\mathsf{Bayes}^2)$ with
\begin{align*}
\mu_{\mathsf{Bayes}} &= \frac{1}{\frac{1}{0.25}+\frac{20}{4}}\left(\frac{2}{0.25}+\frac{\sum_{i=1}^{20}\widehat Y_i}{4}\right)\quad \text{and} \quad 
\sigma_{\mathsf{Bayes}}^2 = 4 + \left(\frac{1}{0.25}+\frac{20}{4}\right)^{-1}.
\end{align*}
More specifically, $Y$ can be expressed as $\mu_{\mathsf{Bayes}} + A+B$, where $A$ and $B$ are both zero-mean Gaussian random variables with $A$ having variance $4$ (because $\PP{Y|Z}=\mathcal{N}(Z,4)$) and $B$ having variance $(1/{0.25}+{20}/{4})^{-1}$, see \cite{murphy2007conjugate} for a detailed derivation of the posterior mean with conjugate prior.
It follows that $\widehat \mu_{\rm MMSE}=\mu_{\mathsf{Bayes}}$. By Theorem~\ref{thm:equivalence:algo:1}, we expect the output $\widehat \mu_{\rm NN}$ of the trained neural network to approximate the minimizer $\widehat \mu_{\rm MMSE}$ of the expected posterior loss, which is biased towards 2 due to the chosen prior. On the other hand, $\widehat \mu_{\rm ERM}$ is an unbiased estimator of the mean, and we thus expect it to behave differently than the other two estimators. 
This simple experiment empirically verifies Theorem~\ref{thm:equivalence:algo:1}.

\subsection{Details on the Newsvendor Experiment in Section~\ref{sec:exp-newsvendor}}
We set the number of possible demand levels to $d=11$, the wholesale price to $p=5$ and and the retail price to $q=7$. We further assume that an observation $X$ consists of $N=20$ historical demand samples (i.e.,~$X=[\widehat Y_1,\ldots,\widehat Y_{20}]$). Finally, we define the ambiguity set $\mathcal U(X)$ in the DRO problem~\eqref{eq:dro} as the family of all demand distributions whose Kullback-Leibler divergence with respect to the empirical distribution on the samples in~$X$ does not exceed~$0.25$. In this case problem~\eqref{eq:dro} can be solved efficiently via the convex optimization techniques developed in \cite{ben2013robust}. Instead of using Algorithms~\ref{algo:training},~\ref{algo:training-erm} and \ref{algo:training-dro} directly, we train the neural networks via the Adam optimizer~\cite{kingma2015adam}.  Training proceeds over $1{,}000$ iterations using batches of $1{,}000$ samples of $Z\sim\mathbb P_{Z}$ and $5$ samples of $(X,Y)\sim \mathbb P_{(X,Y)|Z}$ per iteration. This corresponds to $K=5\times 10^6$ samples of $(X,Y,Z)$ in total as described in Section~\ref{sec:exp-newsvendor}. When generating training samples, we assume that $\mathbb P_Z$ represents the uniform distribution on the 11-dimensional probability simplex. 
Note that this uniform distribution coincides with the Dirichlet distribution of order $d=11$ whose $11$ parameters are all equal to~$1$. As the prior $\mathbb P_Z$ is a Dirichlet distribution, the posterior $\mathbb P_{Z|X}$ is also a Dirichlet distribution with new parameters updated by the observation~$X$ \cite{berger2013bayesian_opt}. Thus, the posterior Bayes action map (the ``True'' strategy) can be computed in closed form. The out-of-sample profit $-\mathbb{E}[\ell(Y,m(X))]$ of any decision strategy $m(X)$ is evaluated empirically on $10{,}000$ test samples of $(X,Y,Z)$. To assess the advantages and disadvantages of the different decision strategies, we consider several test distributions. These test distributions are constructed exactly like the training distribution but use different Dirichlet parameters for~$\mathbb P_Z$. We expect that the performance of the ERM and DRO strategies is immune to misspecifications of the prior $\mathbb P_Z$. The Bayesian strategy and its neural network approximation, however, are expected to suffer under a biased prior. The test performance shown in Figure~\ref{sfig:NW:correct:prior} is evaluated under the correct prior (i.e., all parameters of the Dirichlet distribution~$\mathbb{P}_Z$ are equal to~1).  Figures~\ref{sfig:NW:incorrect:prior1} and~\ref{sfig:NW:incorrect:prior2-appendix} show the impact of a distribution shift. Specifically, the Dirichlet parameters of the test distribution underlying  Figure~\ref{sfig:NW:incorrect:prior1} are set to $0.1$ for the five lowest demand levels and to $2$ for the 6 highest demand levels. Thus, the prior used for training gives too much weight to low demand levels. 
Similarly, the Dirichlet parameters of the test distribution underlying Figure~\ref{sfig:NW:incorrect:prior2-appendix} are set to $2$ for the 6 lowest demand levels and to $0.1$ for the 5 highest demand levels. Thus, the prior used for training gives too much weight to high demand levels.
\begin{figure}
    \centering
    \includegraphics[width=0.5\linewidth]{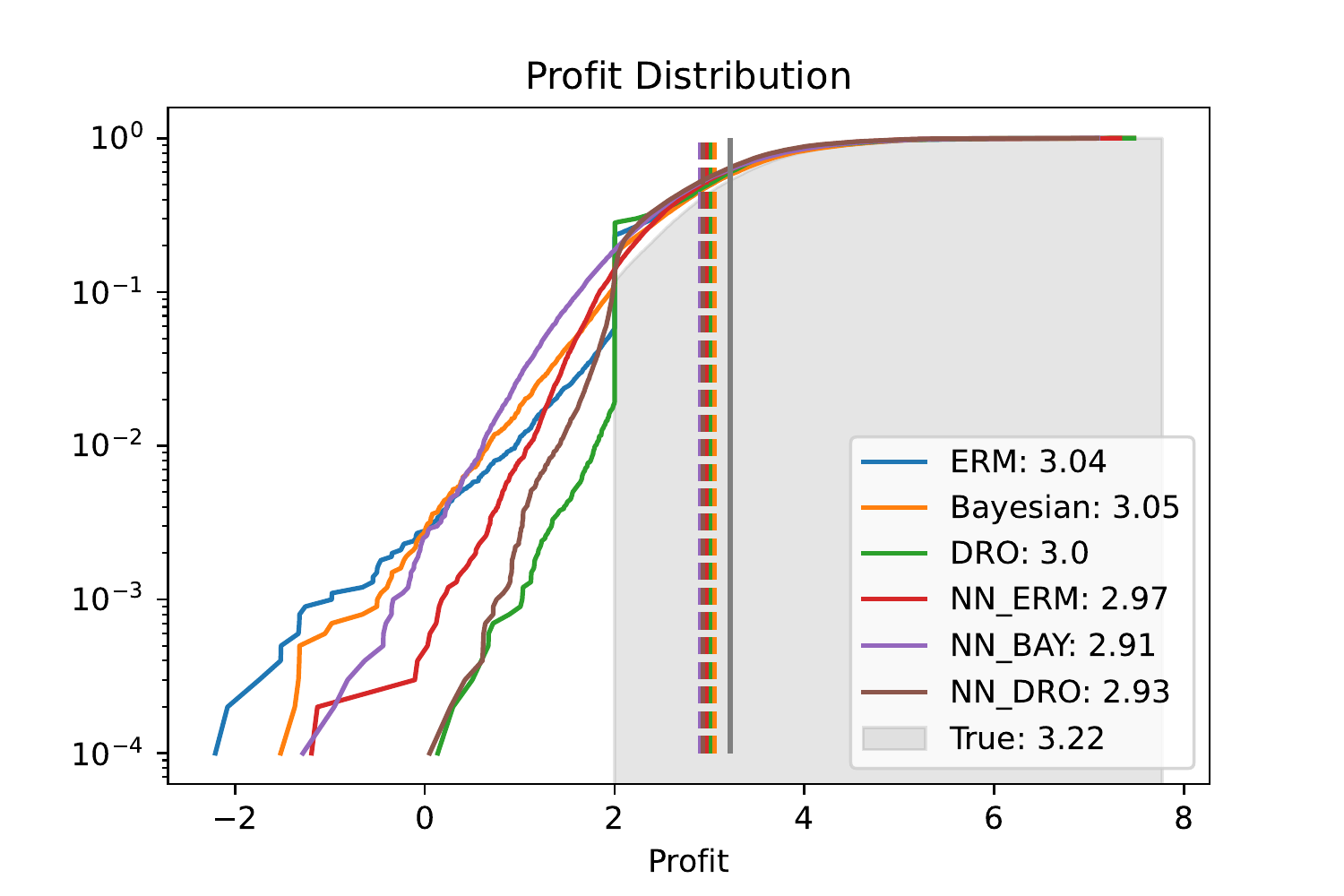}
\caption{Cumulative distribution functions (solid lines) and expected values (dashed lines) of the out-of-sample profit generated by different data-driven strategies for the newsvendor problem (additional example with incorrect prior).
    }
    \label{sfig:NW:incorrect:prior2-appendix}
    \label{fig:cdf_comparison-appendix}
\end{figure}

Figure~\ref{fig:news-decision} compares the decisions obtained with the different methods. In this experiment, the training and testing distributions match, and we set the Dirichlet parameters of the prior $\mathbb{P}_Z$ to $0.5$ for the 7 lowest demand levels and to $2$ for the 4 highest demand levels. 
This particular prior is chosen because it clearly exposes the differences between different strategies. We generate $10^5$ random observations $X\sim\mathbb P_X$ and compute the corresponding decisions. Each chart in Figure~\ref{fig:news-decision} compares one of the exactly computed decisions (vertical axis) against a neural network-based approximation (horizontal axis) trained with Algorithms~\ref{algo:training}, \ref{algo:training-erm} and~\ref{algo:training-dro}. The resulting point clouds visualized  in Figure~\ref{fig:news-decision} are consistent with our theoretical results. That is, the neural network-based approximations align best with the exactly computed decisions in the three charts on the diagonal. Additionally, we observe that the Bayesian strategies order more than the ERM strategies because demand distributions with a large expected value are more likely under the chosen prior. In contrast, the DRO strategies order less than the ERM strategies due to the embedded ambiguity aversion, which favors conservative decisions.

\begin{figure}
    \centering
    \includegraphics[width=\linewidth]{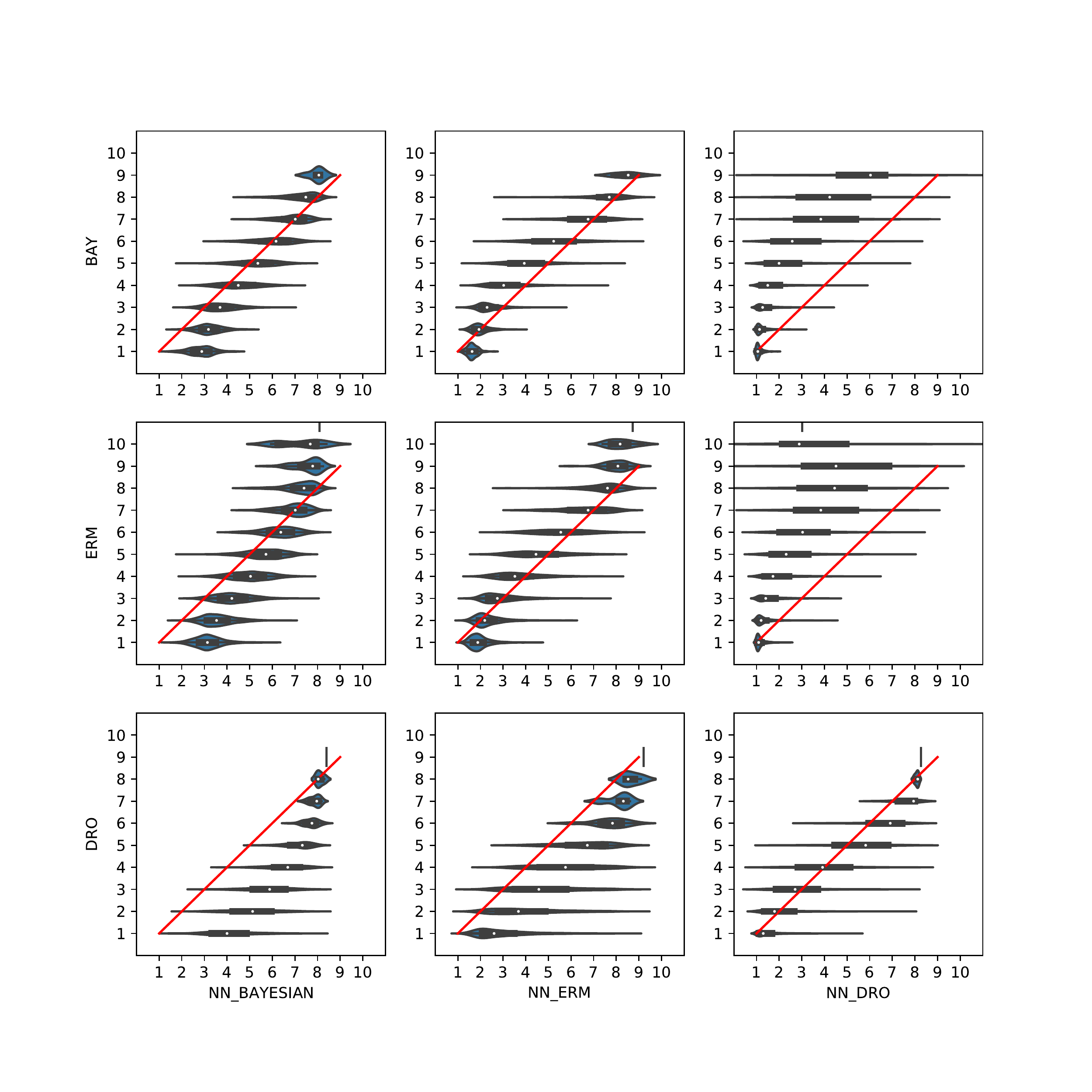}
    \caption{Comparison of the Bayesian, ERM, and DRO strategies against the corresponding neural network-based approximations.}
    \label{fig:news-decision}
\end{figure}

\subsection{Details on the Economic Dispatch Experiment in Section~\ref{sec:exp-economicdispatch}}

We assume that the constant energy demand $d=4$ must be covered by the uncertain output~$Y$ of the wind turbine and by the outputs~$a_j$, $j=1,\ldots,6$, of the six controllable generators. The capacity of the wind turbine equals 2. Thus, at least 2 units of energy must be produced by conventional generators. The capacities and generation costs of these generators are listed in Table~\ref{tab:generators}. The wind turbine produces energy for free but cannot be controlled. A dataset of historical wind power production and weather records with a 10~minute resolution is available from Kaggle.\footnote{\url{https://www.kaggle.com/datasets/theforcecoder/wind-power-forecasting}} 
The dataset covers the period from 1~January~2018 to 30~March~2020. After removing corrupted samples, the period from 1~January~2018 to 31~December~2019 comprises $59{,}532$ records, which we use as the training set. The remaining records are used for testing. We traverse the test set in steps of 10, 30 and 60 minutes to simulate different sampling frequencies. 
For each interval between two consecutive time steps we solve the economic dispatch problem described in Section~\ref{sec:endtoend}.

\begin{table}[]\centering
\caption{Generation costs and capacities of the generators}
\label{tab:generators}\vspace{0.1cm}
\begin{tabular}{l|cccccc}
Generator $j$        & 1       & 2       & 3       & 4       & 5       & 6       \\
\hline
Generation Cost $c_j$& 15 & 20 & 15 & 20 & 30 & 25 \\
Capacity $\Bar{a}_j$& 1     & 0.5   & 1     & 1     & 1     & 0.5  
\end{tabular}
\end{table}

\paragraph{Maximum Likelihood Estimation (MLE)} The MLE approach first uses least squares regression on the training data to construct a prediction~$\widehat Y$ of the wind energy production~$Y$. This prediction is then used as an input for the deterministic prescription problem $\min_{a\in\mathcal A}\ell(\widehat Y,a)$, which outputs the MLE decision. If $Y$ can be expressed as a linear function of the observation~$X$ with an additive Gaussian error, then least squares regression is indeed equivalent to MLE. While MLE outputs an unbiased prediction~$\widehat Y$, the task loss caused by a prediction error $Y-\widehat Y$ is misaligned with the regression loss. Indeed, if $\widehat Y$ overestimates $Y$, then the MLE decision produces too little energy, which incurs high costs of $p=100$ per unit of unmet demand. Conversely, if $\widehat Y$ underestimates $Y$, then the MLE decision produces too much energy. However, this incurs a cost of at most 30 per unit of surplus, that is, the unit production cost of generator~5.

\paragraph{End-to-End (E2E)} 
We compare multiple neural network architectures.

\textbf{(OPL):} The OPL architecture consists of a feature extractor that maps the observation~$X$ to a prediction~$\widehat Y$ of~$Y$ and a prescriptor that maps~$\widehat Y$ to a decision. The feature extractor involves one hidden layer with 64 neurons and ReLU activation functions and an output layer with 1 neuron and a Softplus activation function. The prescriptor subsequently solves the deterministic economic dispatch problem $\min_{a\in\mathcal A}\ell(\widehat Y,a)$, which outputs the E2E decision. 

\textbf{(CAL):} The CAL architecture consists of a feature extractor that maps the observation~$X$ to a 6-dimensional feature~$R$ and a prescriptor that maps $R$ into the feasible set $\mathcal A$. The feature extractor involves one hidden layer with 64 neurons and ReLU activation functions and an output layer with 6 neurons and Sigmoid activation functions, which determine the output of each generator as a percentage of its capacity. A simple rescaling with the generator capacities then yields a decision in~$\mathcal A$. 

Table~\ref{tab:edp-results-appendix} reports the out-of-sample costs of all data-driven decision strategies corresponding to different observations and sampling frequencies. It repeats the results of Table~\ref{tab:edp-results} but also shows the out-of-sample costs that can be earned by observing~$X$ only every 60 minutes. These costs are uncertain for two reasons: (1) the neural network weights are randomly initialized, and (2) the training dataset is shuffled before training. We report the mean as well as the standard deviation of the average cost on the test data over 5 replications of the experiment.

\begin{table*}[t]
\centering
\caption{Mean and standard deviation (in parentheses) of the average test costs generated by different data-driven strategies for the economic dispatch problem.}
\label{tab:edp-results-appendix}
\vspace{0.1cm}
\scalebox{1}{
\begin{tabular}{ll|lll}
Approach                          & Observation & 10 minute frequency & 30 minute frequency & 60 minute frequency \\\hline\hline
\multirow{2}{*}{Baseline}         & Oracle            & 60.688              & 60.691              & 60.703              \\
                                  & Lag-1         & 64.959              & 67.705              & 70.714              \\
\hline\multirow{3}{*}{MLE}              & Myopic            & 281.436(1.652)      & 280.103(0.75)       & 278.686(0.504)      \\
                                  & Myopic Incomp.    & 348.414(0.0)        & 348.425(0.0)        & 348.488(0.0)        \\
                                  & Historical        & 304.558(11.472)     & 277.736(2.378)      & 277.541(5.637)      \\
\hline\multirow{3}{*}{E2E-CAL}          & Myopic            & 68.251(3.73)        & 66.566(3.068)       & 69.565(3.727)       \\
                                  & Myopic Incomp.    & 72.616(0.003)       & 74.107(2.997)       & 74.611(3.999)       \\
                                  & Historical        & 67.088(4.502)       & 77.06(5.585)        & 79.827(3.879)       \\
\hline\multirow{3}{*}{E2E-OPL-Relu}     & Myopic            & 72.601(0.0)         & 72.601(0.0)         & 72.603(0.0)         \\
                                  & Myopic Incomp.    & 72.601(0.0)         & 72.601(0.0)         & 72.603(0.0)         \\
                                  & Historical        & 72.601(0.0)         & 72.601(0.0)         & 72.603(0.0)         \\
\hline\multirow{3}{*}{E2E-OPL-Softplus} & Myopic            & 71.326(1.872)       & 72.527(0.129)       & 69.312(2.673)       \\
                                  & Myopic Incomp.    & 72.604(0.006)       & 72.602(0.002)       & 72.606(0.005)       \\
                                  & Historical        & 72.601(0.0)         & 72.601(0.0)         & 72.603(0.0)        
\end{tabular}}
\end{table*}
\end{document}